\documentclass[11pt]{amsart}

\usepackage{latexsym,amsmath,amscd,amssymb,epsfig,verbatim}

\usepackage{mathdots}
\usepackage{tikz}
\usepackage{graphicx}
\usepackage{subfigure}

\RequirePackage{color}
\definecolor{myred}{rgb}{0.75,0,0}
\definecolor{mygreen}{rgb}{0,0.5,0}
\definecolor{myblue}{rgb}{0,0,0.65}

\RequirePackage{ifpdf}
\ifpdf
  \IfFileExists{pdfsync.sty}{\RequirePackage{pdfsync}}{}
  \RequirePackage[pdftex,
   colorlinks = true,
   urlcolor = myblue, 
   citecolor = mygreen, 
   linkcolor = myred, 
   pagebackref,
   bookmarksopen=true]{hyperref}
\else
  \RequirePackage[hypertex]{hyperref}
\fi

\allowdisplaybreaks
\theoremstyle{plain}
  \newtheorem{theorem}{Theorem}[section]
  \newtheorem{proposition}[theorem]{Proposition}
  \newtheorem{lemma}[theorem]{Lemma}
  \newtheorem{corollary}[theorem]{Corollary}
  
\theoremstyle{definition}
  \newtheorem{definition}[theorem]{Definition}
  \newtheorem{example}[theorem]{Example}

 \theoremstyle{remark}
  \newtheorem{remark}[theorem]{Remark}

\numberwithin{equation}{section}

\newcommand{\NN}{\mathbb{N}}
\newcommand{\CC}{\mathbb{C}}
\newcommand{\Lg}{\mathfrak{g}}
\newcommand{\Lh}{\mathfrak{h}}
\newcommand{\Lp}{\mathfrak{p}}
\newcommand{\Ln}{\mathfrak{n}}
\newcommand{\rank}{\operatorname{rank}}
\newcommand{\Orbit}{\mathcal{O}}

\newcommand{\0}{\Orbit}
\newcommand{\nilcone}{\mathcal N}
\newcommand{\funnil}{\CC[\nilcone]}
\newcommand{\Ht}{\operatorname{ht}}
\newcommand{\levi}{\mathfrak {l}}
\newcommand{\Lb}{\mathfrak{b}}
\newcommand{\Hom}{\mbox{Hom}}
\newcommand{\tr}{\mbox{tr}}
\newcommand{\Sq}{R_+^2}
\newcommand{\Pos}{R_+}

\newcommand{\adj}[1]{[#1]}

\begin{document}

\title
[Equations for some nilpotent varieties]
{Equations for some nilpotent varieties}

\author{Ben Johnson and Eric Sommers}

\dedicatory{To the memory of Bert Kostant}

\address{Department of Mathematics\\
Oklahoma State University\\
Stillwater, OK 74078}
\email{ben.johnson12@okstate.edu}

\address{Dept. of Mathematics and Statistics\\
University of Massachusetts\\
Amherst, MA 01003}
\email{esommers@math.umass.edu}

\begin{abstract}
Let $\0$ be a Richardson nilpotent orbit in a simple Lie algebra $\Lg$ of rank $n$ over $\mathbb C$, 
induced from a Levi subalgebra 
whose $s$ simple roots are orthogonal, short roots.  
The main result of the paper is a description of a minimal set of generators of the ideal 
defining $\overline \0$ in $S \Lg^*$.  In such cases, 
the ideal is generated by bases of either one or two copies of the representation whose highest weight is the dominant short root, 
along with $n-s$ fundamental invariants of  $S \Lg^*$.
This extends Broer's result for the subregular nilpotent orbit, which is the case of $s=1$.  

Along the way 
we give another proof of  Broer's result that $\overline \0$ is normal.
We also prove a result relating a property of the invariants of $S \Lg^*$
to the following question: 
when does a copy of the adjoint representation in $S \Lg^*$ belong to the ideal in 
$S \Lg^*$ generated by another copy of the adjoint representation together with the invariants of $S \Lg^*$? 
\end{abstract}

\maketitle

\section{Introduction}   \label{intro}
\subsection{} Let $G$ be a connected simple algebraic group over $\mathbb C$ with Lie algebra $\Lg$.
Let $S \Lg^*$ be the coordinate ring of $\Lg$, $R:= (S \Lg^*)^G$ its subring of invariants, and $\Pos \subset R$ the invariants without constant terms.  By Chevalley
$R$ is a polynomial ring in $n$ generators where $n$ is the rank of $G$.
Let $f_1,f_2 \dots, f_n$ be a set of fundamental invariants of $R$, that is, a set of homogeneous generators of $R$, with 
$\deg f_1 \leq \dotsc \leq \deg f_n$.  The degrees $d_1 \leq \dotsb \leq d_n$ of these invariants are called the degrees of $G$ and the exponents of $G$ are the numbers $m_i := d_i-1$ for $1 \leq i \leq n$.

Let $\mathcal N$ denote the variety of nilpotent elements in $\Lg$ and let $\mathbb C[\mathcal N]$ denote the regular functions on $\mathcal N$.  Since $\mathcal N$ is closed under scalars, $\mathbb C[\mathcal N]$ is graded and there is a graded surjection $S^i \Lg^* \to \mathbb C^i[\mathcal N]$ for each $i \in \mathbb N$.   Kostant  \cite{Kostant2} showed that the ideal in $S \Lg^*$ defining $\mathcal N$ is $(\Pos) = (f_1, \dots, f_n)$.

Let $B$ be a Borel subgroup of $G$ and $T \subset B$, a maximal torus, with Lie algebras $\mathfrak{b}$ and $\Lh$, respectively.  Let $\Phi \subset X^*(T)$ be the roots of $G$ inside the character group of $T$.  Let $\Pi= \{ \alpha_1, \dots, \alpha_n \} \subset \Phi^+$ be the simple and positive roots determined by the Borel subgroup {\it opposite} to $B$.    For $\alpha \in \Phi$, let $\alpha^\vee \in X_*(T)$ denote the corresponding coroot inside the cocharacters of $T$
and $\langle \cdot, \cdot \rangle$ the pairing on $X^*(T) \times X_*(T)$.  
For $\lambda \in X^*(T)$, let $\Ht(\lambda) := \sum_{i=1}^n c_i$ where $\lambda = \sum_{i=1}^n c_i  \alpha_i$. 
Such a $\lambda$ is called dominant if 
$\langle \lambda,\alpha^{\vee} \rangle \geq 0$ for all $\alpha \in \Pi$, in which case, let $V_{\lambda}$ 
denote the corresponding simple module for $G$ or $\Lg$ with highest weight $\lambda$.    By
our convention on $B$, the $\lambda$-weight space in $V_\lambda$ is stable under the opposite Borel
and $H^0(G/B, \mathbb C_\lambda) \simeq V_\lambda$ in the notation of \S \ref{coho}. 
Let $\theta$ (resp. $\phi$) be the dominant long (resp. short) root of 
$\Phi^+$.
For a parabolic subgroup $P$ containing $T$, let $X^*(P) \subset X^*(T)$ denote the characters of $P$.
Let $W$ be the Weyl group with respect to $T$ and let $s_{\alpha}$ be the reflection corresponding to $\alpha \in \Phi$.

\subsection{} The subregular nilpotent orbit $\0_{sr}$ is the unique nilpotent orbit in $\Lg$ of dimension equal to $\dim \nilcone -2$.  
Broer described the ideal defining ${\overline \0}_{sr}$ in $\Lg$.   

\begin{theorem}[Theorem 4.9 in \cite{Broer1}]
The ideal defining ${\overline \0}_{sr}$ in $\CC[\nilcone]$ is minimally generated by a basis of the unique copy of
 $V_{\phi}$ in $\mathbb C^{\Ht(\phi) }[\nilcone].$  

The ideal defining ${\overline \0}_{sr}$ in $S \Lg^*$ is minimally generated by $f_1, \dots, f_{n-1}$, together with a basis for any copy of $V_{\phi}$ which has nonzero image in $\mathbb C^{\Ht(\phi) }[\nilcone]$.
\end{theorem}

The main result of this paper is a similar description of the ideal defining $\overline \0$ in $S \Lg^*$ for certain 
Richardson orbits studied in \cite{Broer2}.   
Given $\Theta \subset \Pi$, define $\levi_{\Theta}$ to be corresponding Levi subalgebra containing 
$\Lh$, and let $\Lp_\Theta$ be the parabolic subalgebra containing both $\levi_{\Theta}$ and $\Lb$, with $P_\Theta$ the corresponding subgroup
of $G$.
Let $\Ln$ denote the nilpotent radical of $\mathfrak{b}$ and let $\Ln_\Theta$ be the nilradical of $\Lp_\Theta$.
So by our choice of $B$, the root spaces of $\Ln$ correspond to the negative roots. 
Let $\0_{\Theta}$ be the Richardson nilpotent orbit in $\Ln_\Theta$, so that $\0_{\Theta}$ is the unique nilpotent orbit such that $\0_{\Theta}  \cap \Ln_\Theta$ is dense in $\Ln_\Theta$.   This paper is concerned with the situation where $\Theta$ consist of orthogonal short simple roots, which we now assume unless otherwise specified.

\subsection{} 
\label{gen_exps1}
The generalized exponents $m^{\lambda}_i$ of $V_{\lambda}$
are defined by the equation 
$$\sum_{j \geq 0} \dim \Hom_G(V_{\lambda}, \mathbb C^j[\nilcone]) t^j = \sum_{i=1}^k t^{m^{\lambda}_i},$$
where by Kostant \cite{Kostant2} we have 
\begin{equation}  \label{multiplicity}
k = \dim V_{\lambda}^T
\end{equation}
for the number of generalized exponents,
where the superscript denotes the $T$-invariants.
Also from {\it loc.\ cit.}\ the generalized exponents of $V_{\theta}$ are the usual exponents $m_i$ defined above.
Indeed, given an invariant $f \in \Pos$ and a basis $\{ x_i \}$ of $\Lg$, the derivatives 
$\{ \frac{\partial f}{\partial x_i} \}$ span a copy of $V_{\theta}$ in $S \Lg^*$, independent of the choice of basis, and moreover,
the images of the derivatives of the chosen fundamental invariants $f_i$
are a basis for the $V_{\theta}$-isotypic component of $\mathbb C[\nilcone]$.  We are also interested in the generalized exponents for $V_\phi$.  Let $r = \dim V_{\phi}^T$, which equals the number of short roots in $\Pi$ and let
$m^{\phi}_1 \leq \dots \leq m^{\phi}_r$ be the generalized exponents for $V_{\phi}$.
%
%
%
Of course, if $\Lg$ is simply-laced these coincide with the usual exponents.  In Appendix \ref{non-simply}
and \ref{kostant-shapiro},  we recall two ways to determine the $m^{\phi}_i$ for non-simply-laced types.

When $\Theta$ consist of orthogonal short simple roots,  Broer \cite{Broer2} showed that $\overline \0_{\Theta}$ is a normal variety and 
the map $G \times^P \Ln_\Theta \to \overline \0_\Theta$ is birational (we give another proof of these facts in \S \ref{coho}).  Hence by a result of Borho-Kraft (see \cite{Broer2} for the graded version) the analogue of Kostant's result \eqref{multiplicity}
is 
\begin{equation} \label{general_mult}
\dim V_{\lambda}^{\levi_\Theta} = \dim \left( \Hom_G(V_{\lambda}, \CC[\overline \0_{\Theta}]) \right).
\end{equation}
The invariants on the 
left-side of \eqref{general_mult}
are easy to compute for $V_\theta$ and $V_\phi$ since these representations only have roots as non-zero weights.
Setting $s = | \Theta |$,  the left-side becomes $n-s$ and $r-s$, respectively, since $\Theta$ consists only of short simple roots.

Let $I_\Theta$, respectively $J_{\Theta}$, be the ideal defining $\overline \0_{\Theta}$ in $\mathbb C[\nilcone]$, respectively $S \Lg^*$.
Already then we know that there are $s$ independent copies of $V_\phi$ (and of $V_\theta$) in $\funnil$ which lie in $I_\Theta$.  
The main result of the paper is that either one or two copies of $V_\phi$ are needed to generate $I_\Theta$.  We need one more definition before stating the result precisely.

\subsection{} Outside of types $D_{n}$ and $E_7$, given our assumption on $\Theta$, 
there is only one orbit $\0_\Theta$ for any given value of $s = |\Theta|$.
In type $E_7$, there are two orbits with $s=3$.  In type $D_{n}$, there are two orbits  when 
$s=2, 3, \dots, \lceil n/2 \rceil -1$, with partitions $[2n\!-\!2s\!-\!1, 2s\!+\!1]$ and $[2n\!-\!2s\!+\!1,2s\!-\!3,1,1]$, 
and three orbits when $n$ is even and $s=n/2$, because of the two very even orbits with partition $[n,n]$,
together with the orbit with partition $[n\!+\!1,n\!-\!3,1,1]$. 

To state the theorem uniformly, we designate two families of orbits among the orbits we are considering.
For $e \in \0_\Theta$, complete $e$ to an $\mathfrak{sl}_2$-triple $\{e, h, f\}$ with $h \in \Lh$ dominant.
We assign $\0_\Theta$ to the first family if
\begin{equation}\label{first_family}
m^{\phi}_{r-s+1} > \phi(h)
\end{equation}
and to the second family, otherwise.
It turns out that there are at most two values of $\phi(h)$ for a given $s$ and when there are two values, 
the smaller one always satisfies the inequality \eqref{first_family} and the larger one does not.  Also, when there is one orbit for a given $s$,
it satisfies the inequality and hence lies in the first family.
A calculation shows that the second family consists of $\0_\Theta$ with Bala-Carter label $E_6$ in type $E_7$ 
or with partition $[2n\!-\!2s\!+\!1,2s\!-\!3,1,1]$ for $2 \leq s \leq  n/2$ or $[n,n]$ in type $D_n$.  
See Figure \ref{fig1} for some examples.  Inequality \eqref{first_family} will also be relevant for the proof of the theorem (\S \ref{exc_types}).

\begin{figure}[h!]
\caption{The studied nilpotent varieties with second family in red}\label{fig1}
\begin{center}
\begin{subfigure}[In types $D_4$, $D_5$, $D_6$, and $D_7$]{
\begin{tikzpicture}
    \node  (a) at (0,0) {$[7,1]$};
    \node  (b) at (0,-1) {$[5,3]$};
    \node  [color=red] (c) at (-1.9,-2) {$[4,4]^1$};
    \node  [color=red] (d) at (-.6,-2) {$[4,4]^2$};
    \node [color=red] (e) at (1,-2) {$[5,1,1,1]$};
    \node (f) at (0,-3) {$[3,3,1,1]$};
    \node (g) at (0,-4) {};
    \draw (a) -- (b) -- (c) -- (f) -- (d) -- (b) -- (e) -- (f);
    \draw [dotted] (f) -- (g);
\end{tikzpicture}
\begin{tikzpicture}
    \node  (a) at (0,0) {$[9,1]$};
    \node  (b) at (0,-1) {$[7,3]$};
    \node  (c) at (-1,-2) {$[5,5]$};
    \node [color=red] (d) at (1,-2) {$[7,1,1,1]$};
    \node (e) at (0,-3) {$[5,3,1,1]$};
    \node (f) at (-1,-4) {};
    \node (g) at (1,-4) {};
    \draw (a) -- (b) -- (c) -- (e) -- (d) -- (b);
    \draw [dotted] (f) -- (e) -- (g);
\end{tikzpicture}
\begin{tikzpicture}
    \node  (a) at (0,0) {$[11,1]$};
    \node  (b) at (0,-1) {$[9,3]$};
    \node  (c) at (-1,-2) {$[7,5]$};
    \node [color=red] (d) at (.8,-2) {$[9,1,1,1]$};
    \node   [color=red]  (e) at (-2.8,-3) {$[6,6]^1$};
    \node   [color=red] (f) at (-1.5,-3) {$[6,6]^2$};
    \node  [color=red]  (g) at (0,-3) {$[7,3,1,1]$};
    \node (h) at (-1,-4) {$[5,5,1,1]$};
    \node (j) at (-1,-4.8) {};
    \node (k) at (1,-4) {};
    \draw (a) -- (b) -- (c) -- (e) -- (h) -- (f) -- (c) -- (g) -- (d) -- (b);
    \draw (g) -- (h);
    \draw [dotted] (g) -- (k);
    \draw [dotted] (h) -- (j);
\end{tikzpicture}
\begin{tikzpicture}
    \node  (a) at (0,0) {$[13,1]$};
    \node  (b) at (0,-1) {$[11,3]$};
    \node  (c) at (-1,-2) {$[9,5]$};
    \node [color=red] (d) at (1,-2) {$[11,1,1,1]$};
    \node  (f) at (-1,-3) {$[7,7]$};
    \node [color=red] (g) at (1,-3) {$[9,3,1,1]$};
    \node (h) at (0,-4) {$[7,5,1,1]$};
    \node (i) at (-1,-4.8) {};
    \node (j) at (1,-4.8) {};
    \node (k) at (1.7,-4) {};
    \draw (a) -- (b) -- (c) -- (f) -- (h) -- (g) -- (d) -- (b);
    \draw (c) -- (g) -- (h);
    \draw [dotted] (g) -- (k);
   \draw [dotted] (i) -- (h) -- (j);
\end{tikzpicture}}
\end{subfigure}

\begin{subfigure}[In types $E_6$, $E_7$, and $E_8$]{
\begin{tikzpicture}
    \node (a) at (0,0) {$E_6$};
    \node (b) at (0,-1) {$E_6(a_1)$};
    \node (c) at (0,-2) {$D_5$};
    \node (d) at (0,-3) {$E_6(a_3)$};
    \node (e) at (0,-4) {};
    \draw (a) -- (b) -- (c) -- (d);
    \draw [dotted] (d) -- (e);
\end{tikzpicture}
\begin{tikzpicture}
    \node  (a) at (0,0) {$E_7$};
    \node  (b) at (0,-1) {$E_7(a_1)$};
    \node  (c) at (0,-2) {$E_7(a_2)$};
    \node  (d) at (-1,-3) {$E_7(a_3)$};
    \node [color=red] (e) at (1,-3) {$E_6$};
    \node  (f) at (0,-4) {$E_6(a_1)$};
    \node (g) at (-1,-4) {};
    \node (h) at (0,-5) {};
    \draw (a) -- (b) -- (c) -- (d) -- (f) -- (e) -- (c);
    \draw [dotted] (d) -- (g) -- (h) -- (f);
\end{tikzpicture}
\begin{tikzpicture}
    \node (a) at (0,0) {$E_8$};
    \node (b) at (0,-1) {$E_8(a_1)$};
    \node (c) at (0,-2) {$E_8(a_2)$};
    \node (d) at (0,-3) {$E_8(a_3)$};
    \node (e) at (0,-4) {$E_8(a_4)$};
    \node (f) at (-1,-4) {};
    \node (g) at (0,-5) {};
    \draw (a) -- (b) -- (c) -- (d) -- (e);
    \draw [dotted] (d) -- (f) -- (g) -- (e);
\end{tikzpicture}}
\end{subfigure}
\end{center}
\end{figure}

%


\begin{definition} 
For $s \geq 1$, set $m_\Theta$ equal to $m^{\phi}_{r-s+1}$ or $m^{\phi}_{\lceil \frac{r}{2} \rceil}$  according to whether 
$\0_\Theta$ is in the first or second family, respectively.
\end{definition}

Our main result, for $\Theta$ consisting of orthogonal short simple roots and $s \geq 1$, is the following.
\begin{theorem}  \label{main_theorem}
The ideals $I_{\Theta}$ and $J_{\Theta}$ are described as follows:
\begin{enumerate}
\item The lowest degree copy of $V_\phi$ in $I_{\Theta}$
occurs in degree $m_\Theta$.  Denote such a copy by $V$.
\item For $s \geq 2$, there is a copy $V'$ of $V_\phi$ in $I_{\Theta}$ in degree  $m^{\phi}_{r-s+2}$, different from $V$.
\item A basis of $V$ minimally generates $I_{\Theta}$, except when $\0_{\Theta}$
\begin{itemize}
\item has Bala-Carter label $E_6(a_3)$ in type $E_6$; $E_7(a_3)$ or $E_6(a_1)$ in type $E_7$;  $E_8(a_3)$ or $E_8(a_4)$ in type $E_8$, or
\item has partition  $[2n\!-\!2s\!+\!1,2s\!-\!3,1,1]$ for $s \geq 3$ in type $D_n$. 
\end{itemize}
In these cases, a basis of $V'$ is also needed to minimally generate $I_\Theta$.
\item $J_\Theta$ is minimally generated by $n\!-\!s$ fundamental invariants and 
any pre-image of a basis of $V$ and also, in the cases in part (3), of a basis of $V'$.
The $n-s$ invariants have degree $d_1, d_2, \dotsc, d_{n-s}$ for orbits in the first family
and $d_1, d_2, \dotsc, \widehat d_{\lceil \frac{n}{2} \rceil}, \dotsc, d_{n-s+1}$ for orbits in the second.
\end{enumerate}
\end{theorem}

Outside of type $D_n$ with $n$ even, there is a unique choice for $V$ and for $V'$ in the statement of theorem.   
In type $D_n$ with $n$ even, where $n=r$, there are two equal exponents: $m_{\frac{n}{2}} = m_{\frac{n}{2}+1} = n-1$, so we will give the precise choice for $V$ when $m_\Theta = n-1$ in \S \ref{explicit generators classical}.   The choice of $V'$ will be unique except when $s = \frac{n}{2}+1$, in which $V$ and $V'$ can be any choices so that $V + V'$ equals the $V_\phi$-isotypic space in $\mathbb C^{n-1}[\nilcone]$.

\subsection{} The proof works by induction on $s$, with the base case of $s=0$ due to Kostant.  The $s=1$ case is Broer's result, with the same proof.  The general case relies on various cohomological statements established in \S \ref{coho}.  To find a minimal set of generators, we prove a result, Proposition \ref{prop:generating_adjoint}, that makes use of recent results related to flat bases of invariants from \cite{DPP}.  

In type $A_n$ the result, with a different proof, is due to Weyman \cite{Weyman2} 
(see also \cite{Weyman1}).   
Here, $\0_\Theta$ depends only $s= | \Theta|$ and the possible orbits have partition type $[n\!+\!1\!-\!s,s]$ with $s < \frac{n}{2}$.
Let $X$ be a generic matrix
of 
$\mathfrak{gl}_{n+1}$.  
That is, $X = (x_{ij})$, where the $x_{ij}$ are $(n+1)^2$ variables.
Then for an integer $k \geq 1$, 
the entries of any matrix power $X^k$ of $X$ span a copy of the adjoint representation of $\Lg \subset \mathfrak{gl}_{n+1}$
and a copy of the trivial representation in $S^k \Lg^*$.
Weyman's result is that 
$J_{\Theta}$ is minimally generated by the entries of $X^{n+1 -s}$ and the fundamental invariants
$\tr(X^i)$ for $2 \leq i \leq n-s$, 
given by the traces of these powers.  Note that the entries of $X^{n+1 -s}$ already contain $\tr(X^{n+1-s})$
in their span, so this agrees with the statement of the theorem.
It is easy enough to see that each entry of $X^{n+1 -s}$ vanishes on $\0_\Theta$ since 
$M^{n+1-s} = 0$ for $M \in \0_\Theta$, 
so the main part is to show that these entries are enough to generate $I_\Theta$.

Our theorem has a similar interpretation in the other classical Lie algebras using the standard matrix representations (\S \ref{explicit generators classical}).   We can also find a matrix interpretation 
in the exceptional groups using the smallest non-trivial irreducible representation of $\Lg$ (\S \ref{explicit generators exceptional} and Appendix \ref{invariants}); this is useful for the applications considered by the first author in \cite{Ben_Thesis}.  For  orbits in the first family, there is an easy argument for finding the copies of $V_\phi$ which lie in $I_\Theta$ (\S \ref{exc_types}).
After establishing the needed cohomological statements in \S \ref{coho}, we use them to locate a sufficient set of generators of $I_\Theta$ in \S
\ref{sufficient}.  In \S \ref{sec:invariants} we prove the results related to invariants to locate a minimal set of  generators of $J_\Theta$, and then give the explicit descriptions
in each case in \S \ref{explicit generators classical}, completing the details of the proof of 
Theorem \ref{main_theorem}.
In \S \ref{more orbits} we find the defining equations for some additional nilpotent varieties in the non-simply-laced types that occur by folding a simply-laced $\Lg$.  In \S \ref{covariants} we determine which $\Theta$ have a $P_{\Theta}$-covariant of weight $\phi$; this relies on some direct calculations in \S \ref{direct calculations}.   
In Appendices \ref{short_exponents} and \ref{invariants}, we record some results about finding explicit invariants.

\section{Cohomological Statements}  \label{coho}

Let $\Omega$ be a set of orthogonal short simple roots.  
The proof of the main theorem makes use of results about the cohomology groups
$$H^i(G/P_{\Omega},S^{\bullet} \Ln^*_\Omega \otimes \CC_\lambda) \simeq 
H^i(G/B,S^{\bullet} \Ln^*_\Omega \otimes \CC_\lambda),$$
where $\lambda \in X^*(P_{\Omega})$ and $\CC_\lambda$ is the corresponding 
one-dimensional representation of $P_{\Omega}$.  We refer to Jantzen \cite{jantzen} for definitions.
Note that $\lambda \in X^*(P_{\Omega})$ means $\langle \lambda, \alpha^\vee \rangle = 0$ 
for all $\alpha \in \Omega$.
To simplify notation, 
for $m \in \mathbb Z$, 
we follow Broer and write $$H^i_\Omega(\lambda)[-m]$$ 
to refer to the graded $G$-module, with grading by $j \in \mathbb Z$,
\begin{equation*}
\bigoplus_{j \in \mathbb Z} H^i(G/B,S^{j - m} \Ln^*_\Omega \otimes \CC_\lambda).
\end{equation*}

The starting point of the proof of the theorem is that the Springer resolution
$G \times^P \Ln \to \nilcone$ is birational and $\nilcone$ is normal so that
$\CC^i[\nilcone] $ is isomorphic as $G$-module to $H^0(G/B,S^i \Ln^*) $.   Then, as in Broer's proof, we show that 
$\overline \0_\Theta$ is cut out 
from some $\overline \0_\Omega$ with $|\Omega| = s-1$ 
by an ideal whose graded $G$-module structure equals
that of $H^0_\Omega(\phi)[-m]$, 
for some degree shift $m$.

We need to show at various stages that the higher cohomology of some of these modules vanishes.  
One case that is known for general $P_{\Omega}$ is Theorem 2.2 in \cite{Broer2}:  Let $\lambda \in X^*(P_{\Omega})$ be dominant, then
\begin{equation} \label{dominant_vanishing}
H^i_\Omega(\lambda) = 0 \text{ for } i >0,
\end{equation}
where we leave off the graded shift when $m=0$.

By a sequence of cohomological moves, we can sometimes prove vanishing for mildly non-dominant $\lambda \in X^*(P_{\Omega})$.  A full account for the case of $\Omega = \emptyset$ is given by \cite[Theorem 2.4]{Broer1}.
The basic move, which we call the $A_1$-move, goes back to  Demazure
\cite{Demazure} and is the basis for the $\Omega = \emptyset$ result in \cite{Broer1}. 
A general $A_k$-move is the subject of \cite{Sommers:cohomology} 
and this result gets used to prove the normality of nilpotent varieties in type $E_6$
\cite{Sommers1} and of the very even nilpotent varieties in type $D_{n}$ with $n$ even \cite{Sommers:very_even}.   

We now summarize the three moves needed in this paper.  For simplicity,
we assume $\lambda~\in~X^*(P_{\Omega})$, that is, 
$\langle \lambda,\alpha^{\vee} \rangle= 0$ for all $\alpha \in \Omega$ and, as above, that $\Omega$ consists of orthogonal simple short roots,
but these results hold more generally.

\begin{proposition}
Let $\beta \in \Pi$ with $\langle \lambda,\beta^{\vee} \rangle=-1$.
\begin{enumerate}
\item If $\beta$ is orthogonal to all roots in $\Omega$, then 
$s_{\beta}(\lambda) = \lambda+\beta$ and
$$H^i_\Omega(\lambda) \simeq H^i_\Omega(\lambda+\beta)[-1] \text{ for all } i \geq 0 \text{ \ \  ($A_1$ move)}.$$
\item Let $\beta_1 \in \Omega$ be such that $\beta_1, \beta$ determine an $A_2$-subsystem
and such that $\langle \lambda,\beta_1^{\vee} \rangle= 0$.
If $\beta_1$ and $\beta$ are both orthogonal to all roots in $\Omega \backslash \{ \beta_1 \}$,
then 
$s_{\beta_1} s_{\beta}(\lambda) = \lambda+\beta + \beta_1$ and
$$H^i_\Omega(\lambda) \simeq H^i_{\Omega'}(\lambda + \beta+ \beta_1)[-1]  \text{ for all } i \geq 0 \text{ \ \  ($A_2$ move)},$$ 
where $\Omega':= s_{\beta_1} s_{\beta}(\Omega) = (\Omega\setminus\{\beta_1\})   \cup \{ \beta\}$. 
\item   
Let $\beta_1, \beta_2 \in \Omega$ be such that $\beta_1, \beta, \beta_2$ determine an $A_3$-subsystem
and such that $\langle \lambda,\beta_1^{\vee} \rangle= \langle \lambda,\beta_2^{\vee} \rangle = 0$.
If $\beta_1$, $\beta$, and $\beta_2$ are each orthogonal to all roots in $\Omega \backslash \{ \beta_1, \beta_2 \}$,
then 
$s_{\beta} s_{\beta_2} s_{\beta_1} s_{\beta}(\lambda) = \lambda+\beta_1+2 \beta+\beta_2$ and
$$H^i_\Omega(\lambda) \simeq H^i_\Omega(\lambda+\beta_1+2 \beta+\beta_2)[-2] \text{ for all } i \geq 0 \text{  \ \   ($A_3$ move)}.$$ 
\end{enumerate}
\end{proposition}

There is also an $A_2$-move, as in (2), when instead $\langle \lambda,\beta^{\vee} \rangle= 0$.
Namely, \begin{equation} \label{A2_zero_weight}
H^i_\Omega(\lambda) \simeq H^i_{\Omega'}(\lambda)  \text{ for all } i \geq 0.
\end{equation}
This is useful to change between associated parabolic subgroups.

The main result of this section is the following.
\begin{theorem}   \label{cohom_theorem1}
Let $\lambda \in X^*(P_\Omega)$ be a short positive root.
Then for some $m \in \NN$
$$H^i_\Omega(\lambda) \simeq H^i_{\Omega'}(\phi)[-m] \text{ for all } i \geq 0,$$
where $\Omega' = w(\Omega)$ and $\phi = w(\lambda)$ for some $w \in W$.
In particular, $\Omega'$ consists of orthogonal short simple roots and $\phi \in X^*(P_{\Omega'})$.
\end{theorem}

\begin{proof}
If $\lambda \neq \phi$, then there exists $\beta \in \Pi$ such that $\langle \lambda, \beta^{\vee} \rangle = -1$ since $\lambda$ is short.  Clearly, $\beta \not\in \Omega$.  Also $\Lg$ is simple and not of type $A_1$ since $\lambda \neq \phi$, so 
$\beta$ is connected in the Dynkin diagram to at least one other and at most three other simple roots, denoted $\beta_1, \dots, \beta_k$ ($1 \leq k \leq 3$).  We will treat the three cases separately:

\begin{enumerate}
\item   $k=3$.  Then $\Pi$ contains a $D_4$ subsystem with $\beta$ as the central node and $\beta_1, \beta_2, \beta_3$ as the outer nodes.  If $\beta_i \in \Omega$ for all $i$,
 then $\lambda$ is $W$-conjugate to $\lambda + 2(\beta_1 + 2\beta+\beta_2+\beta_3)$. 
 Since $\lambda$ and $\gamma := \beta_1 + 2\beta+\beta_2+\beta_3$ are both short roots (since, e.g., $\beta_1$ is short),
 results on root strings 
 show this is impossible unless $\lambda = -\gamma$ since such a root string can only consist of two roots. 
But the latter is ruled out since $\lambda$ is positive.  
 
If say $\beta_1, \beta_2 \in \Omega$ and $\beta_3 \not \in \Omega$, then we can use the $A_3$-move, with the $A_3$ comprising $\beta_1, \beta, \beta_2$ since  any simple root connected to this $A_3$ cannot be in $\Omega$.  Similarly, if, say, only $\beta_1 \in \Omega$ we can use the $A_2$-move and otherwise the $A_1$-move.

\item $k=2$.   Then $\beta_1, \beta, \beta_2$ form a connected subsystem of rank $3$.  First, we show that if 
$\beta_1 \in \Omega$ or $\beta_2 \in \Omega$, then $\beta$ is also short.
Assume $\beta_1 \in \Omega$ and $\beta$ is long.  Then 
$$s_\beta s_{\beta_1} s_\beta(\lambda) = \lambda + r(\beta + \beta_1)$$ is a root with $r=2$ or $r=3$.
Again by results on root string this is impossible if $\lambda$ is a positive root since $\lambda$ and $\beta+\beta_1$ are both short.  

Consequently, either $\beta$ is long and neither $\beta_1$ nor $\beta_2$ is in $\Omega$ and so 
we can use the $A_1$-move. Or, 
$\beta_1$ and/or $\beta_2$ belong to $\Omega$ and  $\beta$ is short and then we can use the $A_3$-move or $A_2$-move depending on whether or not both are in $\Omega$.

\item $k=1$.  This case is the same as the previous one.
\end{enumerate}

In all cases, we get 
$$H^i_\Omega(\lambda)  \simeq  H^i_{\Omega'}(\lambda')[-m']\text{ for all } i \geq 0,$$
for some positive integer $m'$ and moreover
$\Omega' = x(\Omega)$ and $\lambda' = x(\lambda)$ for some $x \in W$.  In particular,
$\lambda' \in X^*(P_{\Omega'})$.
Since $\Ht(\lambda') > \Ht(\lambda)$, the result follows by induction on height.
\end{proof}


\begin{corollary}   \label{cohom_theorem2}
Suppose there exists $\beta \in  \Pi$ with $\beta$ short and 
$\phi \in  X^*(P_{\Omega \cup \{ \beta \}})$. 
Then there exists $m \in \NN$ and $w \in W$ with
$$H^i_\Omega(\phi +\beta)  \simeq   H^i_{\Omega'}(\mu)[-m] \text{ for all } i \geq 0,$$
where $\mu = w(\phi+\beta)$ is dominant and $\Omega' = w(\Omega)$.
\end{corollary}

\begin{proof}
Consider the subsystem of simple roots orthogonal to $\phi$.  Then $\beta$ and $\Omega$ belong to the simple roots of this
subsystem and satisfy the hypothesis of the previous theorem with $\lambda = \beta$.
Hence working only in the irreducible component of the subsystem containing $\beta$, we have
$H^i_\Omega(\phi +\beta)  \simeq   H^i_{\Omega'}(\phi+\nu)[-m]$
where $m \in \NN$  and $\nu$ is a dominant short root for the subsystem.  Then  
$\mu:=\phi+\nu$ is dominant since $\nu$ is a short root and so has inner product at worst $-1$ at the simple coroots not orthogonal to $\phi$.  The statements about $w$ follow from the previous theorem.
\end{proof}

\section{Finding sufficient generators}  \label{sufficient}

As before, $s = | \Theta |$ where $\Theta$ is a set of orthogonal short simple roots.  
Suppose $s \geq 1$.   Pick an element $\alpha \in \Theta$ and set 
$\Omega := \Theta \backslash \{ \alpha \}$.
Let $I_{\Theta, \alpha}$ be the ideal of $\0_\Theta$ in $\CC[ \overline \0_{\Omega}]$.

\begin{proposition}  \label{main_step}
 $I_{\Theta, \alpha}$ is generated by a basis of a copy of $V_\phi$ in 
  $\CC[ \overline \0_{\Omega}]$. 
\end{proposition}

\begin{proof}
Restricting linear functions on 
$\Ln_{\Omega}$ 
to $\Ln_{ \Theta }$ gives 
the short exact sequence of $B$-modules 
\[0 \to \CC_{\alpha } \to \Ln_{ \Omega}^* \to \Ln_{ \Theta }^* \to 0\]
that has Koszul resolution 
\[0 \to S^{\bullet-1} \Ln_{ \Omega }^* \otimes \CC_{\alpha } \to S^{\bullet} \Ln_{\Omega }^* \to S^{\bullet} \Ln_ { \Theta }^* \to 0,\]
which in turn gives a long exact sequence, which simplifies to
\begin{equation} \label{basic_sequence}
0 \to H^0_{\Omega}(\alpha)[-1] \to H^0_{\Omega}(0) \to  H^0_{\Theta}(0) \to H^1_{\Omega}(\alpha)[-1] \to 0
\end{equation}
since  $H^1_{\Omega}(0) = 0$ by \eqref{dominant_vanishing} for $\lambda = 0$.
By Theorem \ref{cohom_theorem1}  there exists $w \in W$ with
$\Omega':= w(\Omega)$ and a positive integer $m$ such that
$$H^i_{\Omega}(\alpha)[-1]  \simeq H^i_{ \Omega'}(\phi) [-m] \text{ for all $i \geq 0$}.$$  
The latter vanishes for $i >0$ by  \eqref{dominant_vanishing} for $\lambda = \phi$, yielding the exact sequence 
\begin{equation}  \label{surjective1}
0 \to H^0_{ \Omega'}(\phi) [-m]  \to H^0_{\Omega}(0) \to  H^0_{\Theta}(0) \to 0.
\end{equation}
Hence the natural map 
$$H^0_{\Omega}(0) \to  H^0_{\Theta}(0) \text{ is surjective,}$$ 
and so also, by induction on $s$, the map
$$H^0(0) \to  H^0_{\Theta}(0)  \text{ is surjective.}$$
This implies (see \cite{jantzen}) that 
$H^0_{\Theta}(0) \simeq \CC[ \overline \0_{\Theta}]$ and $H^0_{\Omega}(0) \simeq  \CC[ \overline \0_{\Omega}]$
since $H^0(0) \simeq \CC[ \mathcal N]$, and also incidentally that  $\overline \0_{\Theta}$ is normal,
giving a variant of the proof given in \cite{Broer2}.
We conclude from \eqref{surjective1} that 
$H^0_{\Omega'}(\phi)[-m] \simeq I_{\Theta, \alpha}$ as $G$-modules.

Next, consider the sequence of restrictions
$$S\Lg^* \to S\Ln^* \to S \Ln_{\Omega'}^* $$ and 
then the sequence 
$$S \Lg^* \otimes \CC_{\phi}  \to S\Ln^* \otimes \CC_\phi \to S \Ln_{\Omega'}^* \otimes \CC_\phi.$$
Taking global sections over $G/B$ and using that 
$$H^0( G/B, S \Lg^* \otimes \CC_{\phi} ) \simeq S \Lg^* \otimes H^0(G/B,  \CC_{\phi} )  \simeq S\Lg^* \otimes V_\phi,$$
we get maps of $G \times S\Lg^*$-modules
$$S\Lg^* \otimes V_{\phi}  \to H^0(\phi) \to H^0_{\Omega'}(\phi).$$
By 
Broer \cite[Proposition 2.6]{Broer1} 
the first map is surjective since $\phi$ is dominant.  

To get the surjectivity of the second map,
we repeat the first part of the proof, but now with respect to $\Omega'$
and $\Omega'' := \Omega' \backslash \{ \beta \}$ for some choice of $\beta \in \Omega'$.
Then \eqref{basic_sequence} with $\alpha$ replaced by $\beta$
and tensoring with $\CC_\phi$ yields the exact sequence
 \[0\to H^0_{\Omega''}(\phi+\beta)[-1] \to H^0_{\Omega''}(\phi) \to H^0_{\Omega'}(\phi) \to H_{\Omega''}^1(\phi+\beta)[-1] \to 0.\] 
Then Corollary \ref{cohom_theorem2} and \eqref{dominant_vanishing} imply the $H^1$ term vanishes and thus
$H^0_{\Omega''}(\phi) \to H^0_{\Omega'}(\phi)$ is surjective.  Finally by induction 
on $|\Omega'|$ we deduce that 
$S\Lg^* \otimes V_{\phi}  \to H_{\Omega'}^0(\phi)$ is surjective.

Thus the copy of $V_\phi$ in $I_{\Theta, \alpha}$ in degree $m$ generates $ I_{\Theta, \alpha} \simeq H_{\Omega'}^0(\phi)[-m]$.
\end{proof}

\begin{corollary}[see also \cite{Broer2}]
The variety $\overline \0_{\Theta}$ is normal and the map $G \times^P \Ln_\Theta \to \overline \0_\Theta$ is birational.
\end{corollary}

\begin{proof}
As mentioned in the proof, the normality result follows from the 
surjectivity of $H^0(0) \to  H^0_{\Theta}(0)$.  This also implies the birationality statement.  
\end{proof}

\begin{corollary} \label{cor: generators}
The ideal $I_{\Theta}$ defining $\overline \0_\Theta$ in $\mathcal N$ is generated by 
$s$ independent copies of $V_\phi$.
\end{corollary}
 
\begin{proof}
This follows from Proposition \ref{main_step} and induction on $s$.
\end{proof}

As noted in \S\ref{intro}, there are exactly $s$ independent copies of $V_\phi$ in $I_\Theta$.  Another way to see this is that each
$H_S^0(\phi)[-m]$ for $S \subset \Theta$ contains a single copy of $V_\phi$, using Kostant's multiplicity formula and the 
vanishing of $H_S^i(\phi)[-m]$ for $i >0$.  

\section{Invariants}  \label{sec:invariants}

By Corollary \ref{cor: generators} we know $s$ copies of $V_\phi$ will generate $I_{\Theta}$.
In this section we show that at most two copies of $V_\phi$ are needed and moreover
that $n-s$ fundamental generators are further needed to minimally generate $J_\Theta$.  
It turns out that the question of when one copy of $V_\phi$ or $V_\theta$ lies in 
an ideal in $S \Lg^*$ generated by another copy is related to a property of invariants from
De Concini, Papi, Procesi \cite{DPP}, which is related to flat bases of invariants \cite{Saito}. 
On the other hand, the classical types of $A_n$, $C_n$ and most cases in $D_n$ 
could also be resolved using Appendix \ref{short_exponents}, as we explain later.


Pick a basis of $\{ x_i \}$ with $1 \leq i \leq N$ of $\Lg$ and a dual basis $\{ y_i \}$ with respect to the Killing form $( \cdot , \cdot )$.  
When needed, we identify $\Lg \simeq \Lg^*$ using the Killing form.
Let $p$ and $q$ be two homogeneous invariants of degree $a+1$ and $b+1$, respectively.  Then (1)
$p \circ q := \sum_i \frac{\partial p}{\partial x_i} \frac{\partial q}{\partial y_i} $ is again an invariant, homogeneous of degree $a+b$; and (2) 
the span of the $\frac{\partial p}{\partial x_i}$ (or the $\frac{\partial q}{\partial y_i}$) gives a copy of the adjoint representation in $S\Lg^*$.  We write 
$\adj{p}$
for this copy.  


\begin{lemma}   \label{gives_adjoint}
The polynomials
\begin{align*}
w_j := \sum_{i=1}^N 
\frac{\partial^2 p}{ \partial x_j \partial x_i} \frac{\partial q}{\partial y_i},
\end{align*}
with $1 \leq j \leq N$, if nonzero,
span a copy of the adjoint representation in $S\Lg^*$. 
\end{lemma}

\begin{proof}
Let $\pi \in \mbox{Hom}( \Lg \otimes \Lg, S^2 \Lg)$ be the defining homomorphism of $S^2 \Lg$ .
Then $\pi$ can be thought of as element of  $ \mbox{Hom}( \Lg, S^2 \Lg \otimes \Lg^*) \simeq   \mbox{Hom}( \Lg \otimes \Lg, S^2 \Lg)$.
Then the image of $\pi(\Lg)$ under the map
$S^2 \Lg \otimes \Lg^* \simeq S^2 \Lg^* \otimes \Lg^*  \to S^{a+b-3} \Lg^*$, determined by $p$ and $q$, has basis given by the $w_j$.
\end{proof}

Following \cite{DPP}, a homogeneous element $g \in \Pos$ is called a 
{\bf generator} 
if it is not an element of the ideal $\Sq$ in $R$
and 
we write $p \equiv q$ if $p -q \in \Sq$ for $p, q \in R$.
We use $\overline U$ for the image in $\CC[\mathcal N]$ of a subset $U \subset S \Lg^*$.

\begin{proposition}  \label{prop:generating_adjoint}
Let $f_k$ and $f_l$ 
be two fundamental invariants.
Let $U= \adj{f_k}$ and $V=\adj{f_l}$. 
Define $\mathcal I= (\overline V)$, an ideal in $\CC^{m_l}[\mathcal N]$.  
Define $\mathcal J= (V, \{ f_i \ | \ d_i < d_k \} )$, an ideal  in $S \Lg^*$. 

The following are equivalent:
\begin{enumerate}
\item There exists a generator $p$ 
such that $p \circ f_l \equiv f_k$.
\item  $\overline U$ lies in $\mathcal I$.
\item  $f_k \in \mathcal J$.
\end{enumerate}
For any of the equivalent statements to hold, it is necessary by (1) for 
$m_k - m_l + 2 = d_k-d_l+2 $ to be a degree of $\Lg$ since this quantity is $\deg(p)$
and $p$ is a generator.
\end{proposition}

\begin{proof}
$(1)\Rightarrow (2)$.  Suppose there exists a generator $p$ with $p \circ f_l \equiv f_k$.  Then $\deg(p) = m_k - m_l +2$.
Let 
$$w_j=\sum_{i=1}^N  \frac{\partial^2 p}{\partial x_j \partial x_i}\frac{\partial f_l}{\partial y_i},$$  which by 
Lemma \ref{gives_adjoint} is a basis for a copy of the adjoint representation in $S^{m_k} \Lg^*$.  
Now 
\begin{align}  \label{invariant condition}
\sum_{j=1}^N x_j w_j = 
\sum_{j=1}^N x_j \left( \sum_{i=1}^N  \frac{\partial^2 p}{\partial x_j \partial x_i}\frac{\partial f_l}{\partial y_i} \right) &  \nonumber \\
=  \sum_{i=1}^N \left(  \sum_{j=1}^N   x_j  \frac{\partial}{\partial x_j } \left(  \frac{\partial p}{\partial x_i} \right)\right) \frac{\partial f_l}{\partial y_i}   = 
\sum_{i=1}^N \left(  (m_k - m_l +1) \frac{\partial p}{\partial x_i}  \right) \frac{\partial f_l}{\partial y_i}  & =  (m_k - m_l +1) p \circ f_l,
\end{align}
by Euler's formula since the $ \frac{\partial p}{\partial x_i}$ are homogeneous of degree $m_k - m_l +1$.  

Next, from Kostant's fundamental description  \cite{Kostant2} of $S \Lg^*$ as $R \otimes H$, 
where $H \subset S \Lg^*$ is a graded subspace isomorphic as $G$-module to $\CC[\mathcal N]$, 
and the fact that the $\adj{f_j}$ are a basis for 
the $V_{\theta}$-isotypic component in $\CC[\mathcal N]$, we know
that every homogeneous copy $T$ of the adjoint representation in $S \Lg^*$ is the span of elements 
\begin{equation} \label{basis_v}
v_i:= \sum_{j=1}^n r_j \frac{\partial f_j}{\partial x_i}
\end{equation}
for $r_1, \dots, r_n \in R$ homogeneous (and independent of $i$).
The choice of $r_j$ is unique up to a scalar that is independent of $j$.  
In particular, $T$ lies in $(\Pos)$, the ideal in $S \Lg^*$ generated by $\Pos$ 
if and only if all $r_j$ are  of positive degree.  
Since $\sum_i  x_i v_i =  \sum_j d_j r_j f_j$, we deduce that $T \subset (\Pos)$ if and only if 
there exists a basis of $v'_i$ of $T$ with $\sum_i  x_i v'_i \in \Sq$.   Since there is only one copy of the trivial 
representation in $\Lg^* \otimes \Lg \simeq \Lg \otimes \Lg$, any such invariant $\sum_i  x_i v'_i$ is well-defined up to a scalar.
From $p \circ f_l \equiv f_k$, we conclude from \eqref{invariant condition} and the fact that $f_k$ is a generator
that at least one $w_j$, and hence all $w_j$, are not  contained in $(\Pos)$,
and moreover the span of the images of the $w_j$ in $\mathbb C^{m_k}[ \mathcal{N}]$ 
coincides with that of both $\adj{p \circ f_l}$ and $U$. 
Since clearly each $w_j \in \mathcal I$, this concludes the proof of this implication.

$(2) \Rightarrow  (3)$.  Suppose $\overline U \subset \mathcal I$.  Then certainly $U \subset \mathcal J$.
Hence $f_k =  \frac{1}{d_k} \sum_j x_j \frac{\partial f_k}{\partial x_j}$  by Euler's formula and so $f_k \in \mathcal J$.

$(3) \Rightarrow  (1)$.
There is a graded surjection 
$$S \Lg^*[-m_l] \otimes \Lg \to (V),$$
where $(V)$ refers to the ideal in $S \Lg^*$, 
and so any invariant in $(V)$ comes from an element in $(S \Lg^*[-m_l]~\otimes~\Lg)^G$.
But $(S \Lg^*[-m_l] \otimes \Lg)^G \simeq \Hom(\Lg, S \Lg^*[-m_l])$, and thus
any invariant in $(V)$ 
takes the form $\sum_i v_i \frac{\partial f_l}{\partial y_i}$ with some $v_i$ as in \eqref{basis_v}.
Since $f_k \in \mathcal J$, it follows that $f_k \equiv \sum_i v_i \frac{\partial f_l}{\partial y_i}$
for some such $v_i$.
Since $f_k$ is a generator, $\sum_i v_i \frac{\partial f_l}{\partial y_i}$ is therefore a generator, 
and thus at least one of the $r_j$ in the expansion of the $v_i$'s must be degree $0$.  
Hence 
$f_k \equiv \sum_i \frac{\partial p}{\partial x_i} \frac{\partial f_l}{\partial y_i}$ for $p$ a linear combination of the chosen fundamental invariants.
In particular, $p$ is a generator.  
\end{proof}

It is known from \cite{DPP} 
when condition (1) in the Proposition holds; it can be checked by restricting to a Cartan subalgebra of $\Lg$.
To state the result, we need to take the matrix representation of $\mathfrak{so}_{2n}$ and a generic matrix $X \in \mathfrak{so}_{2n}$\footnote{That is, $X = \sum_{i=1}^N x_i v_i$, where the $v_i$ are a $\CC$-basis of $\mathfrak{so}_{2n}(\CC)$ and the $x_i$ are formal variables.}.  Then a set of fundamental invariants of $\mathfrak{so}_{2n}$ are given by the Pfaffian $\mbox{Pf}(X)$ and 
$\tr(X^{2i})$ for $1 \leq i \leq n-1$.


\begin{theorem}[\cite{DPP}]   \label{flat_basis}
Condition (1) in the Proposition is equivalent  to 
\begin{equation} \label{degree_condition}
d_k - d_l + 2 \text{ is equal to a degree of } G,
\end{equation}
except when $G$ has type $D_n$ with $n$ even and one of the following is true:
\begin{itemize}
\item If $f_l \equiv c \cdot \mbox{Pf}(X)$ for some $c$, then condition (1) holds only when $d_k = 2n -2$. 
\item If $d_k=n$ and $f_k \not \equiv c \cdot \mbox{tr}(X^{n})$ for all scalars $c$, then condition (1) holds only when $d_l=2$. 
\end{itemize}
\end{theorem}

In particular, condition (1) always holds when $k=n$, that is, $f_k=f_n$ has maximal degree $d_n$, equal to the Coxeter number.
For types $A_n, B_n, C_n$ and most cases in $D_n$, it is straightforward 
to show that $[f_k] \subset ([f_l])$ in $\funnil$ if and only if $d_k \geq d_l$ and hence 
even deduce Theorem \ref{flat_basis} from Proposition \ref{prop:generating_adjoint}, rather than the other way around
(see \S \ref{short_exponents}).
In the others cases we use Proposition \ref{prop:generating_adjoint}  and Theorem \ref{flat_basis} to study such ideals:

\begin{proposition}  \label{two_copies}
An ideal in $\funnil$ generated by multiple copies of $V_\theta$ and $V_\phi$ is minimally generated by bases of at most two such copies.
Let $I \subset \funnil$ be an ideal of the above type.  
Then the preimage of $I$ in $S \Lg^*$ is minimally generated by at most two copies of $V_\theta$ and/or $V_\phi$ together with $n-t$ fundamental invariants, where $t$ is number of independent copies of $V_\theta$ in $I$. 
\end{proposition}

\begin{proof}
By Proposition \ref{prop: simply-laced to non}, the first statement for non-simply-laced groups follows from the statement for the simply-laced ones.  We now check the first statement in all possible cases in the simply-laced types.   In type $A_n$ any such ideal is generated by its copy in lowest degree, either by the direct matrix argument (\S \ref{gl}) or the fact that $d_k - d_l + 2$ is always a degree.   
In type $D_n$, by Theorem \ref{flat_basis}, any such ideal is 
generated by the restriction of the entries of $X^{2i+1}$ for some $i$, or the restriction of the derivatives of $\mbox{Pf}(X)$, or both.
When $n$ is even, there are also the ideals generated  
by the restriction of the derivatives of $\tr(X^n) + c \mbox{Pf}(X)$ for any $c \neq 0$, alone or together with 
the restriction of the entries of $X^{2i+1}$ for some $i$ with $2i+1<n-1$.
In the exceptional types the copies of $V_\theta$ in $\funnil$
occur in unique degrees and we label them by the corresponding degree (that is, by an exponent).  
In type $E_6$, the only such ideals not generated by a single copy are
$(V_4, V_5)$, $(V_5, V_7)$ and $(V_7, V_8)$.
In type $E_7$,   
they are $(V_5, V_7)$,  $(V_7, V_9)$, $(V_9, V_{11})$, $(V_{11}, V_{13})$.
In type $E_8$,  
they
are $(V_7, V_{11})$, $(V_{11}, V_{13})$, $(V_{13}, V_{17})$,   $(V_{17}, V_{19})$, $(V_{19}, V_{23})$.  This is shown by checking whether or not \eqref{degree_condition} holds. 

The statement about minimal generators in $S \Lg^*$ is clear from Proposition \ref{prop:generating_adjoint} when there is a single copy of $V_\theta$ generating $I$.  There is also a version of Proposition \ref{prop:generating_adjoint}  
where $V=[f_l]$ with $\epsilon(f_l)= -f_l$ (see \S \ref{non-simply}) 
is replaced by its restriction to $\Lg_0$, yielding a copy of $V_\phi$.  Then the implication (3) implies (1) still holds, where
the derivatives in (1) are with respect to a basis of $\Lg_0$.  So the result holds also when there is a single copy of $V_\phi$ 
from Proposition \ref{prop: simply-laced to non} and the simply-laced case.  In the remaining cases, where the ideal is generated by two representations, say $V_j$ and $V_i$ with $j \geq i$, we check that any fundamental invariant of degree at least $j+1$ is already contained in either $(V_i)$ or $(V_j)$.  The result follows.
\end{proof}

\begin{remark}
Analogous statements for Propositions \ref{prop:generating_adjoint} and \ref{two_copies}
hold for the reflection representation $V_\theta^T$ and the irreducible representation $V^T_\phi$ of the Weyl group $W$.
\end{remark}

\section{Proof of Theorem \ref{main_theorem} and explicit generators} \label{explicit generators exceptional}  \label{explicit generators classical}


The classical cases can mostly be handled using the material from Appendix \ref{short_exponents}.  
We take the standard matrix representation of each classical $\Lg$ and  let $X$ be a generic matrix.  From Appendix \ref{short_exponents}, in types $A, C, D$,
the span of the entries of $X^{m^{\phi}_i}$, restricted to $\funnil$, give distinct copies of $V_\phi$ and outside of type $D$, these are all the copies.
We now check by hand which ones vanish on $\0_\Theta$ (see also \cite{Richardson} or \S \ref{exc_types}).
The results for $I_\Theta$ follow from the following results for $J_\Theta$.   Assume $s \geq 1$.

\subsection{Type $A_n$}  $\0_\Theta$ has partition type $[n\!+\!1\!-\!s,s]$.
By Appendix \ref{gl} the $s$ copies of $V_\phi$ in $I_\Theta$ come from the span of the entries of $X^{n+1-j}$ for $1 \leq j \leq s$.  
Hence, a minimal generating set for $J_\Theta$ is given by the entries of $X^{n+1-s}$ and $\tr(X^2), \dots, \tr(X^{n-s})$, and we recover Weyman's result.   These orbits are in the first family and the theorem follows since $m^\phi_{r-s+1} = n+1-s$.

\subsection{Type $D_n$}
\subsubsection{$\0_\Theta$ has partition type $[2n\!-\!2s\!-\!1, 2s\!+\!1]$ for $1 \leq s < \frac{n}{2}$} 
By Appendix \ref{BD_restrict} the $s$ copies of $V_\phi$ in $I_\Theta$ come from the span
of the entries of $X^{2n-2j-1}$ for $1 \leq j \leq s$.  Hence,  
a minimal generating set for $J_\Theta$ is given by a basis chosen from the entries of $X^{2n-2s-1}$ and $\tr(X^2), \tr(X^4), \dots, \tr(X^{2n-2s-2}),$
and $\mbox{Pf}(X)$.   
These orbits are in the first family and the theorem follows since $m^\phi_{r-s+1} = 2n-2s-1$.

\subsubsection{$\0_\Theta$ has partition type $[2n\!-\!2s\!+\!1, 2s\!-\!3, 1,1]$ for $2 \leq s \leq \frac{n}{2}+1$}  
\label{need Pfaffian}
The entries of $X^{2n-2j-1}$ vanish on $\0_\Theta$ if and only if $1 \leq j \leq s-1$,
 so we are missing one copy of $V_\phi$:   by Appendix \ref{BD_restrict} or Proposition \ref{two_copies},
 it must lie in degree $n$, otherwise the entries of $X^{2n-2s-1}$ would lie in $I_\Theta$.
Indeed, the derivatives of $\mbox{Pf}(X)$ with respect to a basis of $\Lg$  
vanish on $M \in \mathfrak{so}_{2n}$ whenever $\rank(M) \leq 2n-4$ (see \cite{Ben_Thesis}).  
Hence, by Proposition \ref{two_copies}, for $s \geq 3$, a minimal generating set for $J_\Theta$ is given by a basis of the entries of $X^{2n-2s+1}$, 
the derivatives of $\mbox{Pf}(X)$,
and $\tr(X^2), \tr(X^4), \dots, \tr(X^{2n-2s})$.
For $s =2$, a minimal generating set for $J_\Theta$ is given by the derivatives of $\mbox{Pf}(X)$
and $\tr(X^2), \tr(X^4), \dots, \tr(X^{2n-2s})$.
These orbits are in the second family except for when $s = \frac{n}{2}+1$.  The theorem follows since 
$m_{\Theta} = n-1$ and $m_{r-s+2} = 2n-2s+1$ in all cases.


\subsubsection{$\0_\Theta$ has partition type $[n,n]$ with $n$ even}

There are two orbits with partition type $[n,n]$ and $s = \frac{n}{2}$.
We first show that the derivatives of $\frac{1}{2n} \mbox{tr}(X^{n}) + c  \mbox{Pf}(X)$ vanish on 
$\0_\Theta$ for either $c=1$ or $c=-1$.


Let $e \in \0:= \0_\Theta$ and
put $e$ in an $\mathfrak{sl}_2$-subalgebra $\mathfrak s$ spanned by the triple $\{e, h, f\}$. 
Identifying $\Lg$ with the Lie algebra of skew-symmetric matrices $\mathfrak{so}_{2n} \subset \mathfrak{gl}_{2n}$, 
we can consider the matrices $s(x):= e + xf^{n-1}$ with $x$ an indeterminate.  
Since $n$ is even and $f$ is skew-symmetric, 
it follows that $f^{n-1}$ is skew-symmetric and thus $s(x) \in \mathfrak{so}_{2n}$ for each $x \in \CC$.
Now $\CC^{2n}$, viewed as the defining representation for $\mathfrak{so}_{2n}$, 
decomposes under the restriction to $\mathfrak s$
as the direct sum of two copies of the irreducible $n$-dimensional representation $U_n$ 
of $\mathfrak s$.  
Working with respect to 
the basis of $\CC^{2n}$ 
coming from the standard basis of $U_n$,
it is easy to see that $s(1)$ has two repeated diagonal blocks with rational entries and that  
$s(x)^n$ is a multiple $\lambda$ of the identity matrix, with $\lambda =a x$ for some nonzero rational number $a$.
The first property means that $\det (s(1))$ is positive rational, and the second, that $\det (s(x)^n) = (a x)^{2n}$.  
It follows that $\det (s(x)) = (a x)^2$ since the determinant is continuous in $x$ 
and thus $\mbox{Pf}(s(x)) = \pm a x$ since the Pfaffian of a skew-symmetric matrix squares to its determinant.
At the same time, $\mbox{tr}(s(x)^{n}) = 2n a x$.  So let $X = s(x)$.  Then  
the derivative at $x$ of $\frac{1}{2n} \mbox{tr}(X^{n}) + c\mbox{Pf}(X)$ must vanish on $\0$ for either $c=1$ or $c=-1$. 

Next, it is clear that  the entries of $X^{2n-2j-1}$ vanish on 
$\0$ if and only if $1 \leq j < s$. 
As in \S \ref{need Pfaffian} this means the one missing copy of $V_\phi$ is in degree $n$, which we just located.  By 
 Proposition \ref{two_copies}, a minimal generating set for $J_\Theta$ for the two very even orbits is given by 
a basis of the derivatives of $\frac{1}{2n} \mbox{tr}(X^{n})  \pm \mbox{Pf}(X)$ and
$\tr(X^2), \tr(X^4), \dots, \tr(X^{n})$.
The theorem follows since $m_{\Theta} = n-1$ and $m_{r-s+2} = n+1$.


\subsection{Type $B_n$}  The subregular orbit is the only case with $s \geq 1$, already handled by Broer.  
Since $r=1$ there is a unique copy of $V_{\phi}$ in $\funnil$, and is located in degree $n$.  It must therefore cut out the subregular orbit since $s=1$.  
By Appendix \ref{BD_restrict} it is obtained by restricting the derivatives of $\mbox{Pf}(X)$ to the nilcone in $\mathfrak{so}_{2n+1}$, where $X$ is a generic matrix of  $\mathfrak{so}_{2n+2}$.  The minimal set of invariants needed in $J_\Theta$ are just 
$\tr(X^2), \tr(X^4), \dots, \tr(X^{2n-2})$ by Proposition \ref{two_copies} or Broer's simpler argument.

\subsection{Type $C_n$}
$\0_\Theta$ has partition type $[2n\!-\!2s,2s]$ for $1 \leq s < \frac{n}{2}$.
By Appendix \ref{C_restrict} the $s$ copies of $V_\phi$ in $I_\Theta$ come from the span of the entries of $X^{2n-2j}$ for $1 \leq j \leq s$.  
Hence a minimal generating set for $J_\Theta$ is given by a basis of  the entries of $X^{2n-2s}$ and $\tr(X^2), \tr(X^4), \dots, \tr(X^{2n-2s-2})$.

\subsection{Exceptional Types}\label{exc_types}

We deduce the results from Corollary \ref{cor: generators} and Proposition \ref{two_copies} 
and find explicit generators using \S \ref{invariants}.
We can still manage to avoid calculating the degree shift $m$ in Theorem \ref{cohom_theorem1} with a simple method to determine almost all copies of $V_\phi$ that vanish on $\0_\Theta$; this method would also work for classical types.  

As in the introduction $\{ e, h, f \}$ is an 
$\mathfrak{sl}_2$-triple in $\Lg$ with $h \in \Lh$ dominant.  
Kostant's observation  \cite{Kostant2} for $e$ regular carries over more generally (see \cite{Reeder}) and was used by Richardson \cite[Proposition 2.2]{Richardson} for $V_\theta$ in the same way as we do here.
That is, for a highest weight representation $V_{\lambda}$ and its linear dual $V^*_\lambda$, 
recall that $V_\lambda^{G_e}$, where $G_e$ is the centralizer of $e$ in $G$, 
carries a grading via the action of $\frac{1}{2}h$ and this graded space is isomorphic to 
$$\bigoplus_i \Hom_G(V^*_{\lambda}, \mathbb C^{i}[\0_e]).$$
Suppose a copy $V$ of $V^*_\lambda$ lies in degree $m_i^{\lambda}$ in $\funnil$.  We have
\begin{equation}  \label{inequality}
m_i^{\lambda} > \frac{1}{2} \lambda(h) \implies V \text{ vanishes on } \0_e.
\end{equation}
Hence by the definition \eqref{first_family}, for all orbits in the first family, including in the classical types, 
there are $s$ distinct generalized exponents for 
$\lambda = \phi$ satisfying the inequality in \eqref{inequality}, 
namely, $m^\phi_r \geq m^\phi_{r-1} \geq \dots \geq m^\phi_{r-s+1}$.  
This allows us to locate all copies of $V_\phi \simeq V^*_\phi$ in $I_\Theta$ and to show that the copy in lowest degree occurs in degree $m^\phi_{r-s+1}$.

As in Appendix \ref{invariants}, let $U$ be a non-trivial irreducible representation of $\Lg$ of minimal dimension and embed
$\Lg \subset \mathfrak{gl}_N$ for a choice of basis of $U$.  Then taking a generic matrix $X \in \Lg$ the polynomials
$f_i := \tr(X^{d_i})$ are a set of fundamental invariants.  The derivatives of these invariants along $\Lg$ determine $n$ copies of $V_\theta$, which 
are independent on restriction to $\mathcal N$.   We label these copies of $V_\Theta$ as $V_{d_1}, \dotsc, V_{d_n}$ since the degrees are distinct in the exceptional groups.

\subsection{Type $E_6$}  
Every orbit is in the first family so \eqref{inequality} finds all $s$ copies of $V_\phi$ in $I_\Theta$.  
Hence Proposition \ref{two_copies} is enough to complete the proof.
Here and below we take a basis of the listed copies of $V_\phi$ to obtain minimal generators for $J_\Theta$.

\smallskip


\noindent
$E_6(a_1)$: $V_{11}$ and $f_1, f_2, f_3, f_4, f_5$; \ $D_5$: $V_8$ and $f_1, f_2, f_3, f_4$; \ $E_6(a_3)$: $V_{7}$, $V_8$ and $f_1, f_2, f_3$.

\subsection{Type $F_4$ and $G_2$}  
The subregular orbit is the only case with $s \geq 1$ and was treated by Broer. 
Both orbits satisfy \eqref{first_family} 
and so the copy of $V_\phi$ in degree $m^\phi_{r}$ is the unique copy in $I_\Theta$.  
In $F_4$ this $V_\phi$ is obtained by restricting the adjoint representation $V_8$ from $E_6$ to $F_4$ by 
\S \ref{non-simply}.  In $G_2$ the $V_\phi$ is obtained by restricting the unique $V_\phi$ for $B_3$ in degree $3$.
%

\subsection{Type $E_7$}
Except for the $E_6$ orbit, all orbits are in the first family so \eqref{inequality} finds all $s$ copies in $I_\Theta$.  
For the $E_6$ orbit, however, $s=3$ but $\frac{1}{2} \theta(h) = 11$, which implies only that $V_{13}$ and $V_{17}$ are in $I_\Theta$.  The ideal $I_\Theta$ cannot coincide with the ideal for $E_7(a_3)$, which contains $V_{11}$.  Hence by Proposition \ref{two_copies} the only other ideal of the desired kind containing $3$ copies of $V_\theta$ is $(V_9)$. 

\smallskip

\noindent
$E_7(a_1)$:  $V_{17}$ and $f_1,\dotsc,f_6$;  \ \ $E_7(a_2)$:   $V_{13}$ and $f_1,\dotsc,f_5$ \\
$E_7(a_3)$:  $V_{11}$, $V_{13}$ and $f_1,f_2, f_3, f_4$ \ \  $E_6$:    $V_{9}$  and $f_1,f_2,f_3, f_5$  \\
$E_6(a_1)$:  $V_{9}$, $V_{11}$ and  $f_1,f_2,f_3$.

\subsection{Type $E_8$}
Every orbit is in the first family.   
 
 \smallskip

\noindent
$E_8(a_1)$:  $V_{17}$ and $f_1,\dotsc, f_7$; \ \ $E_8(a_2)$:  $V_{23}$ and $f_1,\dotsc,f_6$ \\
$E_8(a_3)$:  $V_{19}$, $V_{23}$ and $f_1,\dotsc,f_5$; \  \ $E_8(a_4)$:  $V_{17}$, $V_{19}$ and $f_1,f_2, f_3,f_4$

\section{Other orbits in non-simply-laced cases}  \label{more orbits}

Using \S \ref{non-simply} and Proposition \ref{two_copies}, 
we can find equations for those nilpotent varieties in a non-simply-laced Lie algebra that lie in one of the orbits $\0_\Theta$ under the embedding 
$\Lg_0$ into $\Lg$. 
That these equations generate the ideal $J$ defining the nilpotent variety requires cohomological results along the lines of Theorem \ref{cohom_theorem1} that will appear elsewhere.   

\subsection{ $[2n-2s+1, 2s-1, 1]$ in type $B_n$ for $s \geq 1$ }
From the $D_{n+1}$ result for 
$[2n\!-\!2s\!+\!1, 2s\!-\!1, 1,1]$, 
the ideal $J$ is minimally generated by 
a basis of the copy of $V_\phi$ in degree $n$, and when $s \geq 2$
a basis of the entries of $X^{2n-2s+1}$ (a copy of $V_\theta$), 
and $\tr(X^2), \dots ,\tr(X^{2n-2s})$.

\subsection{$[n, n]$ in type $C_n$, $n$ odd}
From the $A_{2n-1}$ result for $[n,n]$, the ideal $J$ is generated by a basis of the entries of $X^{n}$ (a copy of $V_\theta$) and
$\tr(X^2), \dots, \tr(X^{n-1})$.

\subsection{$F_4(a_2)$ in type $F_4$}
From the $E_6(a_3)$ case in $E_6$, the ideal $J$ is minimally generated by a copy of $V_\theta$ in degree $7$, a copy of $V_\phi$ in degree $8$, and fundamental invariants in degrees $2$ and $6$.

\section{Covariants} 
\label{covariants}


By running the proof of Theorem \ref{cohom_theorem1} backward in any given case, 
we obtain the following identity for various 
sets $\Omega$ of orthogonal short simple roots with $\phi \in X^*(P_\Omega)$:
\begin{equation} \label{reverse}
H^i_{\Omega}(-\phi)[m] \simeq  H^i_{\Omega'}(-\alpha)[1]   \text{ for all } i \geq 0,
\end{equation}
for some positive integer $m$ and $\alpha$ a simple root orthogonal to the simple roots in $\Omega'$ 
and where $\Omega'$ and $\Omega$ are conjugate by an element in $W$.
Next by \cite{Demazure}, we get
$$H^0_{\Omega'}(-\alpha)[1]   \simeq H^0_{\Omega'}(0).$$
Also the steps of the proof can be repeated with $\lambda=0$, and using \eqref{A2_zero_weight}, to obtain 
$H^0_{\Omega'}(0) \simeq H^0_{\Omega}(0).$
Hence 
\begin{equation}  \label{covariant identity}
H^0_{\Omega}(-\phi)[m] \simeq H^0_{\Omega}(0)
\end{equation}
and unraveling the notation for the lowest degree term in the grading
$$H^0(G/P, S^m (\Ln_{\Omega}^*) \otimes \CC_{-\phi}) \simeq H^0(G/P, \CC) \simeq \CC$$
where $P= P_{\Omega}$.
It follows that 
there exists a 
$P$-equivariant polynomial $\sigma: \Ln_{\Omega} \to \mathbb C_{\phi}$ of homogeneous degree $m$.  
Such an $\sigma$ is called a $P$-covariant.
Any such covariant is unique and its zero set determines a well-defined orbit of codimension two in $\overline \0_\Omega$.
By Theorem \ref{cohom_theorem1} and then \eqref{inequality}, the degree $m$ of $\phi$
satisfies $m \leq \frac{1}{2}\phi(h)$ where $\{e, h, f\}$ is a triple for
$e \in \0_\Omega$ and $h \in \Lh$ is dominant.  But it turns that equality often holds in our cases.
Consider the subspace $\mathfrak j = \oplus_{i \geq 2} \Lg_i$ defined by the action of $\mbox{ad}(h)$.

\begin{proposition} \label{get_the_degree}
Suppose $\mathfrak j \subset \Ln_\Omega$, 
then a $P$-covariant $\sigma$ with weight $\phi$ has degree $\frac{1}{2}\phi(h)$. 
\end{proposition}

\begin{proof}
Consider the cocharacter $\chi: \CC^* \to T$ with $T \subset P$, determined by $h$.   
By our convention that $B$ corresponds to the negative roots, and the hypothesis on $\mathfrak j$,
we can choose the $\mathfrak{sl}_2$-triple so that $f \in \Ln_\Omega$.
Then for $\xi \in \CC^*$, we can evaluate $\chi(\xi).\sigma(f)$ in two ways.  
First, $\chi(\xi).\sigma(f) = \xi^{\phi(h)}\sigma(f)$ since 
$\chi(\xi) \in T$ and $\sigma(f) \in \CC_\phi$.  Second, $\chi(\xi).f = \xi^{-2} f$, so 
$\chi(\xi).\sigma(f) = \sigma(\chi(\xi)^{-1}.f) = \sigma(\xi^{2}f)$.
It follows that the degree of $\sigma$ as a function in $S \Ln^*_\Theta$ is $\frac{1}{2 }\phi(h)$.
\end{proof}

We want to find all $\Omega$, consisting of orthogonal simple short roots 
with $\phi \in X^*(P)$, that can have a $P$-covariant of weight $\phi$. 
It turns out that these are exactly the cases obtained in the manner above, i.e., by reversing 
Theorem \ref{cohom_theorem1}.  Note that such an 
$\Omega$ cannot have the maximal value of $|\Omega|$ for $\Lg$ since
there exists a short simple root orthogonal to all elements in such an $\Omega$.
Moreover, we find for each orbit $\0_\Theta$ with $|\Theta|$ not maximal that there exists $\Omega$ 
with $\0_\Theta = \0_\Omega$ 
having a $P$-covariant of weight $\phi$ satisfying Proposition \ref{get_the_degree}.
The result is

\begin{proposition}  \label{which are covariants}
There exists a $P$-covariant with weight $\phi$ 
for $P=P_\Omega$ with $\phi \in X^*(P)$, 
except in the following cases:
\begin{enumerate}
\item Type $B_n$ with $\Omega = \{ \alpha_n \}$.
\item Type $C_n$ with $\alpha_1 \in \Omega$.   
\item Type $D_n$ with $\{ \alpha_1, \alpha_{n-1}, \alpha_n \} \subset \Omega$. 
\item Type $F_4$ with $\Omega = \{ \alpha_3 \}$.  
\item Type $E_6$ with $\Omega = \{ \alpha_1, \alpha_4, \alpha_6 \}$.  
\item Type $E_7$ with $\Omega  =  \{ \alpha_3, \alpha_5, \alpha_7\}$ 
or $\{ \alpha_2, \alpha_3, \alpha_5, \alpha_7\}$.  
\item Type $E_8$ with $\{ \alpha_2, \alpha_5, \alpha_7 \} \subset \Omega$.
\end{enumerate}
\end{proposition}

\begin{proof}

The obstacle to \eqref{reverse} holding, 
using the proof of Theorem \ref{cohom_theorem1},
is encountering a $D_4$ subsystem of $\Pi$ with the roots corresponding to the 
end nodes lying in $\Omega$
or a $C_2$  subsystem with short root lying in $\Omega$.  
This implies, for example, that if $\Lg$ is simply-laced and $|\Omega| \leq 2$ or $\Lg$ is of type $A_n$, then we never encounter such a subsystem and therefore a covariant always exists.

Now, let $\Pi_\phi$ be the simple roots orthogonal to $\phi$ and $\levi_\phi$ the Levi subalgebra they determine.  We need only consider the equivalences classes of $\Omega \subset \Pi_\phi$ under the action of $W(\levi_\phi)$,
since if one $\Omega$ works in a class, so will all the others using \eqref{A2_zero_weight}.   
By the proof of the main theorem,  there is at least one successful $\Omega$ for each orbit $\0_\Theta$ with $|\Theta| = |\Omega|+1$. 
In fact, it turns out that each equivalence class exactly corresponds to a unique way to cut out an $\0_\Theta$ from 
an $\overline \0_\Omega$.  We now proceed through each case.


In $D_n$, the type of $\levi_\phi$ is $D_1 \times D_{n-2}$ and so 
there are generally going to be four equivalence classes of $\Omega$ for any given $t = |\Omega|$, depending 
on whether $\alpha_1 \in \Omega$ or not, and both  $\alpha_{n-1}$ and $\alpha_n$ are in $\Omega$ or not.  
There are also the extra cases arising 
from when $\Omega$ determines a very even orbit in the factor of $D_{n-2} \subset \levi_\phi$.
Lemmas \ref{D_main_cohom_lemma} and \ref{D_weird_2}  
handle the cases where $\alpha_1 \not \in \Omega$.  
These are the cases such that some $\Omega$ in the equivalence class satisifies Proposition \ref{get_the_degree}.
Next, Lemma \ref{D_weird} and \eqref{symmetric_variant} 
handle the cases with $\alpha_1 \in \Omega$ and $\alpha_{n-1}$ and $\alpha_n$ are not both in $\Omega$.

In $E_6$ there is only one equivalence class when $|\Omega| \leq 2$ since $\levi_\phi$ has type $A_5$
and each of these will therefore be covered by  Proposition \ref{get_the_degree}.  The remaining case, 
when $|\Omega|=3$, is an exceptional case.

In $E_7$, there are several equivalence classes since $\levi_\phi$ has type $D_6$.
When $t=|\Omega| \leq 2$ we know \eqref{reverse} holds since we do not encounter a $D_4$.
But there is one case for $t=2$ not covered by Proposition \ref{get_the_degree} (see Lemma \ref{E7_extra}).
This case arises because there are two different orbits $\0_\Theta$ with $|\Theta|=3$.  
Next, there are three equivalence classes when $t=3$.  Two of them must work since we can obtain $E_8(a_4)$
from either $E_6$ or $E_8(a_3)$.  The calculations are carried out in Lemmas \ref{JM for t=3} and \ref{extra for t=3}
and there is an $\Omega$ in both cases covered by Proposition \ref{get_the_degree}.  The remaining
equivalence class for $t=3$ is an exceptional case.

In $E_8$, there is an extra equivalence class when $t=3$ since $\levi_\phi$ has type $E_7$.  It is covered by Proposition \ref{get_the_degree} (see Lemma \ref{JM for t=3}).

The $C_n$ cases are all covered in Lemma \ref{C_moves} and the type $A_n$ cases come for free (or from Lemma \ref{A_moves}).
We have now covered all cases that are not exceptional cases.

To complete the proof, we check by hand that the other cases encounter a $D_4$ or $C_2$ if we carry out the steps in Theorem 
\ref{cohom_theorem1} 
and this turns out to be enough to show that $\phi$ is not a $P$-covariant, but $2 \phi$ is, using \cite[\S 2.1]{Sommers-Taiwan}.
The fact that this holds is related to the fact that for all the exceptional cases, which include the cases where $|\Omega|$ is maximal, each 
such $\0_\Omega$ has $G$-equivariant fundamental group isomorphic to $S_2$ (for $G$ adjoint).
%
\end{proof}

\begin{remark}
For the exceptional cases listed in the proposition, it is possible to show that 
$H^0_{\Omega}(\phi)[-d]$ equals the $G$-module of sections of the non-trivial $G$-equivariant line bundle on $\0$ when $G$ is adjoint.
Here, $2d$ is the degree of the covariant for $2\phi$.
When $\Ln_\Theta = \mathfrak j$, we also get that $d = \frac{1}{2}\phi(h)$ as in Proposition \ref{get_the_degree}.
When $\phi = \theta$, the value of $d$ is no longer an exponent of $\Lg$, but it can be described as an exponent for a hyperplane arrangement
(see \cite{Sommers-Trapa}, \cite{Broer_exponents}).  For the $E_8$ maximal case, for example, $d = 14$.
\end{remark}

\begin{remark}
Propositions \ref{get_the_degree} and \ref{which are covariants} explain why $\frac{1}{2}\phi(h)$ is often  a generalized exponent for $V_\phi$.  Since the value of $\phi(h)$ must weakly decrease as 
we move down the partial order on nilpotent orbits, it is perhaps not surprising that \eqref{inequality} ends up determining the copies of $V_\phi$ in $I_\Theta$ for orbits in the first family.
\end{remark}

\begin{example}
Consider the subregular orbit in $D_n$.  Here $\Theta$ is a single simple root and when $\Theta \neq \{\alpha_2\}$, then there is a $P$-covariant $\sigma$.  When $\Theta = \{ \alpha_1 \}$, the degree of $\sigma$ is $n-1$, while in all other cases $\sigma$ has degree $2n-5$.  The latter cases include $\Theta = \{ \alpha_{n\!-\!2} \}$, which yields $\Ln_\Theta = \mathfrak j$ and so $2n-5 = \frac{1}{2 }\phi(h)$.   The two different situations correspond to the two different orbits contained in $\overline \0_\Theta$.  The first is the orbit $[2n\!-\!3,1,1,1]$ and the second is $[2n\!-5,\!5]$ (when $n \geq 5$).   The orbits are seen here as $G$-saturation of the zero set of each $\sigma$.
\end{example}


\begin{remark}
The first part of Proposition 2.4 in \cite{Broer2} and Proposition \ref{which are covariants} give a way to avoid 
Corollary \ref{cohom_theorem2} and the inductive steps in the proof of Proposition \ref{main_step} that show that
the ideal determined by $H_\Omega(\phi)$ is generated by $V_\phi$.  
On the other hand, it seems the part of Proposition 2.4 in {\it loc.\ cit.}
related to $\Ht_P$ is incorrect as the above example shows.   
In particular, the result assigned to Richardson is not true:  for example, for the subregular orbit $\0$ in $D_4$ and $\Theta = \{ \alpha_1\}$, then $\0$ does not meet $\Ln_1$.
\end{remark}


\section{Direct calculations}
\label{direct calculations}

\subsection{Helpful lemma}

Let $\mathfrak m$ be a standard Levi subalgebra of $\Lg$ of type $A_k$.
Let $\Pi_{\mathfrak m}= \{ \beta_1, \dots, \beta_k \}$ be the simple roots of $\mathfrak m$
 with $\beta_1$ and $\beta_k$ the two extreme vertices of the Dynkin diagram of $\mathfrak m$.
Let $\mathfrak m'$ be the standard Levi subalgebra containing $\mathfrak m$ whose simple
roots $\Pi_{\mathfrak m'}$ consist of $\Pi_{\mathfrak m}$ and all simple roots adjacent to some 
simple root in $\Pi_{\mathfrak m}$.    Let $\Omega \subset \Pi$ be a set of orthogonal simple roots.


\begin{lemma}  \label{main_lemma}
Assume that $\Omega \cap \Pi_{\mathfrak m'} \subset  \{ \beta_2, \beta_3, \dots, \beta_{k} \}$.  
Let $t = |\Omega \cap \Pi_{\mathfrak m'}|$.

Given $\lambda \in X^*(P_{\{ \beta_2, \dots, \beta_k \} })$
such that
$$\langle \lambda, \beta^{\vee}_1 \rangle = -1,$$
we have
$$H^i_{ \Omega } (\lambda)
\simeq 
H^i_{ \Omega'}(\lambda  \!+\!   \beta_1  \!+\!  \beta_2  \!+\!  \dotsc  \!+\! \beta_k)[-k \!+\! t],$$
 for all $i \geq 0$ and 
for any $\Omega' \subset \Pi$ satisfying:
\begin{enumerate}
\item $\Omega'  \cap \Pi_{\mathfrak m'} \subset  \{ \beta_1, \beta_2, \dots, \beta_{k-1} \}$ consists 
of orthogonal simple roots; 
\item $\Omega'  \cap (\Pi \backslash \Pi_{\mathfrak m'}) = \Omega  \cap (\Pi \backslash \Pi_{\mathfrak m'})$; and 
\item $|\Omega| = |\Omega'|$.
\end{enumerate}
\end{lemma}

\begin{proof}
The proof is by induction on $k$.
If $k=1$, then $t=0$ and the $A_1$-move gives the results with $\Omega' = \Omega$.
Consider general $k \geq 2$.  

If $\beta_2 \not \in \Omega$, then the $A_1$-move gives
$H^i_{ \Omega } (\lambda)
\simeq 
H^i_{ \Omega}(\lambda +  \beta_1)[-1].
$
Now $\langle \lambda \!+\! \beta_1, \beta^{\vee}_2 \rangle = -1$ and 
$\lambda \!+\! \beta_1 \in X^*(P_{\{ \beta_3, \dots, \beta_k \} })$ and so
 we can apply induction to the the Levi subalgebra of type $A_{k-1}$ with simple roots $\beta_2, \dots, \beta_k$ to get 
$$H^i_{ \Omega}(\lambda +  \beta_1)[-1] \simeq H^i_{ \Omega'}(\lambda +  \beta_1 + \beta_2 + \dotsc + \beta_k)[-1 -(k\!-\! 1) \!+\! t]$$
for any $\Omega'$ as in the statement since if needed we can use \eqref{A2_zero_weight} repeatedly
to ensure that $\Omega'$ includes $\beta_1$.

If $\beta_2  \in \Omega$, then the $A_2$-move gives
$$
H^i_{ \Omega } (\lambda) 
\simeq 
H^i_{ \Omega'}(\lambda  \!+\!   \beta_1 \!+\!  \beta_2)[-1],
$$
where $\Omega'$ and $\Omega$ are the same, except that $\Omega'$ includes $\beta_1$ instead of $\beta_2$.
Now $\langle \lambda \!+\! \beta_1 \!+\! \beta_2, \beta^{\vee}_3 \rangle = -1$ and
$\lambda \!+\! \beta_1+\! \beta_2 \in X^*(P_{\{ \beta_4, \dots, \beta_k \} })$ and so
we can apply induction to the the Levi subalgebra of type $A_{k-2}$ with simple roots $\beta_3, \dots, \beta_k$ to get 
$$H^i_{ \Omega}(\lambda  \!+\!  \beta_1  \!+\!  \beta_2 )[-1] 
\simeq H^i_{ \Omega'}(\lambda \!+\!  \beta_1  \!+\!  \beta_2   \!+\!  \beta_3  \!+\!  \dotsc + \beta_k)[-1 -(k \!-\! 2) \!+\! (t \!-\! 1)]$$
for any $\Omega'$ as in the statement since if needed we can use \eqref{A2_zero_weight} to permute around the roots of $\Omega'$ to include $\beta_1$ and $\beta_2$.
In either case, the result follows.
\end{proof}

\subsection{Type $A_n$}
Let $\Omega= \{ \alpha_3 , \dotsc,  \alpha_{2t-1}, \alpha_{2t+1}\}$.
One use of Lemma \ref{main_lemma} gives
\begin{lemma}  \label{A_moves}
For all $i \geq 0$,
$H^i_{ \Omega } (\alpha_1)[-1] \simeq H^i_{ \{ \alpha_2, \alpha_4,  \dotsc,  \alpha_{2t-2}, \alpha_{2t} \} }(\phi)[-n+t].$
\end{lemma}


\subsection{Type $C_n$}
Let $\Omega= \{ \alpha_3 , \dotsc,  \alpha_{2t-1}, \alpha_{2t+1}\}$
with $2t+1 < n$.

\begin{lemma}  \label{C_moves}
For all $i \geq 0$, we have
$H^i_{ \Omega }(\alpha_1)[-1] \simeq H^i_{ \Omega }(\phi)[-2n+2+2t].$ 
\end{lemma}

\begin{proof}
Here, $\phi=\alpha_1+2 \alpha_2+2 \alpha_3+\dotsb+2\alpha_{n-1}+\alpha_n$. 
Use Lemma \ref{main_lemma} applied to the Levi subalgebra $\mathfrak m$ with simple roots $\{\alpha_2, \dots, \alpha_{n-1} \}$
and $\lambda = \alpha_1$ to get
$$H^i_{ \Omega }(\alpha_1)[-1] \simeq 
H^i_{ \{ \alpha_2, \alpha_4,  \dotsc,  \alpha_{2t-2}, \alpha_{2t} \} }(\sum_{i=1}^{n-1}\alpha_i)[-1-(n-2)+t].$$
Next, since $\langle \sum_{i=1}^{n-1} \alpha_i ,\alpha^\vee_n \rangle=-1$, one use of the $A_1$ move yields that the latter
is isomorphic to 
$$H^i_{ \{ \alpha_2, \alpha_4,  \dotsc,  \alpha_{2s-4}, \alpha_{2s-2} \} }(\sum_{i=1}^{n}\alpha_i)[-n+t].$$
Another use of Lemma \ref{main_lemma} applied to $\mathfrak m$, with the ordering of the roots reverse
and with $\lambda = \sum_{i=1}^{n}\alpha_i$ gives the isomorphism with 
$H^i_{ \Omega } (\phi) [-2n+2+2t].$ 
\end{proof}

\subsection{Type $D_n$}  
%


\subsubsection{} 
Let $\Omega= \{ \alpha_3 , \dotsc,  \alpha_{2t-1}, \alpha_{2t+1}\}$.
Let $\Omega'=-w_0(\Omega)$, which is different from $\Omega$ only in the case when $n=2t+2$.
The proof of the next lemma is similar to the previous ones.

\begin{lemma}  \label{D_main_cohom_lemma}
For all $i \leq 0$, we have $H^i_{\Omega}(\alpha_1)[-1] \simeq H^i_{ \Omega' }(\phi)[-2n-3+2t].$ 
\end{lemma}

In the case $n =2t+2$, the proof also works when the roles of $\Omega$ and $\Omega'$ are interchanged.

\subsubsection{}  
Now let $\Omega= \{ \alpha_{n-2t+1},\alpha_{n-2t+3}\dotsc,\alpha_{n-3}, \alpha_{n-1}\}$ with $t \geq 1$.

\begin{lemma} \label{D_weird}
For all  $i \geq 0$, we have $H^i_{ \Omega }(\alpha_n)[-1] \simeq H^i_{ \{\alpha_1,\alpha_{3},\alpha_5, \dotsc,\alpha_{2t-3}, \alpha_{2t-1} \}}(\phi)[-n+1]$.
\end{lemma}

\begin{proof}
When $t \geq 2$, the first $t-1$ moves are type $A_3$, giving 
$$H^i_{ \Omega }(\alpha_n)[-1] \simeq H^i_{ \Omega }(\alpha_{n-2t+1}+2\alpha_{n-2t+2}+\dotsb+2\alpha_{n-2}+\alpha_{n-1}+\alpha_n)[-2t+1].
$$
Then $n-2t$ moves of type $A_2$ are used to add the remaining roots
to get the isomorphism with 
$$H^i_{\{\alpha_1,\alpha_{n-2t+3},\dotsc,\alpha_{n-3}, \alpha_{n-1}\}}(\phi)[-n+1].$$ 
We can then change the parabolic by using \eqref{A2_zero_weight} to shift over the simple roots to get the isomorphism in the statement of the lemma.
\end{proof}

When $n=2t$, we also get a symmetric variant.  
Let $\Omega= \{\alpha_1,\alpha_{3},\alpha_5, \dotsc,\alpha_{n-3} \} \cup \{ \alpha_{n} \}$.
Then 
\begin{equation} \label{symmetric_variant}
H^i_{ \Omega }(\alpha_{n-1})[-1] 
\simeq H^i_{\Omega  }(\phi)[-n+1].
\end{equation}

\subsubsection{}  
Now let $\Omega= \{ \alpha_{n-2t+3}, \alpha_{n-2t+5},\dotsc,\alpha_{n-3}, \alpha_{n-1},  \alpha_{n}  \}$
with $t \geq 2$.  Similar to the previous cases,

\begin{lemma} \label{D_weird_2}
For all  $i \geq 0$, we have
$H^i_{ \Omega }(\alpha_1)[-1] \simeq H^i_{ \Omega }(\phi)[-2n+1+2t]$.
\end{lemma}


\subsection{Others cases where $t=2$ or $t=3$}

\begin{lemma}  \label{E7_extra}
In $E_7$, for all $i \geq 0$, we have
$$H^i_{\{\alpha_2,\alpha_5\}}(\alpha_7)[-1] \simeq H^i_{\{\alpha_2,\alpha_3\}}(\phi)[-9].$$
\end{lemma}

\begin{lemma}  \label{JM for t=3}
In $E_7$ and $E_8$, for $m = 9, 17$, respectively, for all $i \geq 0$, we have
$$H^i_{\{\alpha_3,\alpha_5,\alpha_7\}}(\alpha_2)[-1] \simeq H^i_{\{\alpha_2,\alpha_3,\alpha_5\}}(\phi)[-m].$$
\end{lemma}

\begin{lemma}  \label{extra for t=3}
In $E_7$, for all $i \geq 0$, we have
$$H^i_{ \{\alpha_2,\alpha_5,\alpha_7\}}(\alpha_3)[-1] \simeq H^i_{\{\alpha_2,\alpha_5,\alpha_7\}}(\phi)[-11].$$
\end{lemma}

\appendix
\section{Restricting to subalgebras} 
\label{short_exponents}

\subsection{Non-simply laced Lie algebras}  \label{non-simply}
We state some results, likely already known, 
about the relationship between a simple non-simply-laced $\Lg_0$, i.e., $\Lg_0$ of type $B_n, C_n, G_2, F_4$,
and its associated simple simply-laced Lie algebra $\Lg$, i.e., $D_{n+1}, A_{2n-1}, D_4, E_6$, respectively.   
As is well-known, $\Lg_0$ arises as the invariant space of an automorphism 
$\epsilon: \Lg \to \Lg$, preserving $\Lh$. Moreover, $\epsilon$ can be taken to be induced 
from a diagram automorphism of $\Pi$, also denoted $\epsilon$, and we can pick root vectors $e_\alpha$ for $\alpha \in \Pi$
such that $\epsilon(e_\alpha) = e_{\epsilon(\alpha)}$.  Also $\epsilon$ has order equal to $2$ or $3$.  See \cite{onish-vin:book} for these results.
Let $\Pi_l = \{ \alpha \in \Pi \ | \ \epsilon(\alpha) = \alpha \}$ and $\Pi_s = \Pi \backslash \Pi_l$. 
The key property about $\epsilon$ is that for all $\alpha \in \Pi_s$, we have that $\epsilon(\alpha)$ and $\alpha$ are orthogonal.  Assume for simplicity of statement that $d=2$, 
so that $\Lg = \Lg_0 \oplus \Lg_{1}$ is the eigenspace decomposition under $\epsilon$ for eigenvalues $1$ and $-1$, respectively. 

Then for any such $\epsilon$, it is an exercise to show (in a case-free manner) that
$\Lg_0$ is a Lie algebra with two root lengths,
the non-zero weights of the $\Lg_0$-representation on $\Lg_{1}$ consist of the short roots of $\Lg_0$ 
and the zero weight space  of the $\Lg_0$-representation on $\Lg_{1}$ is exactly $\Lh_{1}$,
and that, in fact, 
\begin{equation} \label{decomp}
\Lg \simeq \Lg_0 \oplus V_{\phi}
\end{equation}
as $\Lg_0$-module, where $\phi$ now refers to the dominant short root of $\Lg_0$.
Moreover, every root $\beta$ of $\Lg$ restricted to $\Lh_0$ is a root $\bar \beta$ of $\Lg_0$ and 
$\Ht(\beta) = \Ht(\bar \beta)$ where the latter is computed using the induced simple roots of $\Lg_0$.
See Proposition 8.3 in \cite{kac}.

Pick an $\epsilon$-stable set of fundamental invariants $f_1, \dots, f_n$ for $\Lg$. 
Let $F_0 = \{ f_i  \ | \ \epsilon(f_i) = f_i \}$ and $F_1 = \{ f_i  \ | \ \epsilon(f_i) =  -f_i \}$.
\begin{proposition}  \label{prop: simply-laced to non}
The following statements hold:
\begin{enumerate}
\item The restriction to $\Lg_0$ of the elements in $F_0$ gives a set of fundamental invariants for $\Lg_0$.
\item The Jacobian of the $f_i$'s remains nonzero upon restriction to the nilcone $\mathcal N_0$ of $\Lg_0$.  
The derivatives of $f_i$ span a copy of $V_\theta$ (resp.,\ $V_\phi$) 
in $\mathbb C[ \mathcal N_0]$ when 
$f_i \in F_0$ (resp.,\ $f_i \in F_1$).
\item The exponents of $\Lg_0$ are the $\deg(f_i)-1$ for $f_i \in F_0$. 
\item The generalized exponents of $V_{\phi}$  are the $\deg(f_i)-1$ for $f_i \in F_1$. 
\end{enumerate}
\end{proposition}

\begin{proof}
First, $\Lg_0$ contains a regular nilpotent element of $\Lg$, namely 
$\sum_{\alpha \in \Pi} e_{\alpha}$.  
Since the $f_i$ are fundamental invariants for $\Lg$, their Jacobian is nonzero when evaluated at any regular nilpotent element of $\Lg$.
Hence, the $n$ vectors of derivatives of the $f_i$'s are linearly independent on restriction to 
the nilcone of  $\Lg_0$.  
Now if $\epsilon(f_i) = f_i$, the derivatives with respect to vectors in $\Lg_{1}$ cannot be $\epsilon$-invariant and so vanish as functions on $\Lg_0$.  That means the derivatives with respect to vectors in $\Lg_{0}$ cannot vanish and so span a copy of the adjoint representation of $\Lg_0$.  
Similarly for $\epsilon(f_i) = -f_i$.  
Hence by the linear independence of the vector of derivatives of the $f_i$'s on the the nilcone of $\Lg_0$
and the fact that $\dim V_{\theta}^T + \dim V_{\phi}^T  = \dim \Lg^T$, we have accounted for all the generalized exponents
of $V_{\theta}$ and $\dim V_{\phi}$ by Kostant's result \eqref{multiplicity}.
\end{proof}

\subsection{Kostant-Shapiro formula for exponents}  \label{kostant-shapiro}
The Kostant-Shapiro formula states \cite{Kostant1} that the exponents for $\Lg$ are obtained 
as the dual partition $\mu$ of the partition of $|\Phi^+|$ given by 
$$\# \{\alpha \in \Phi^+ \ | \ \Ht(\alpha) = j  \}$$ for $j = 1, 2, \dots, \Ht(\theta)$.

We recall that the proof arises by taking a regular nilpotent $e$ and one of its 
$\mathfrak sl_2$-subalgebras $\mathfrak s$ with standard basis $e, h,f$ with $h \in \Lh$ dominant.
Since $\alpha(h) =2$ for $\alpha$ simple, the value $\Ht(\alpha)$ coincides with $\frac{1}{2} \alpha(h)$ 
for any root $\alpha$.   Since $h$ is regular and the values $\alpha(h)$ are even for each root $\alpha$,
the representation of $\mathfrak s$ on $\Lg$ has $n=\dim(\Lh)$ irreducible constituents.
Moreover, the centralizer $\Lg_e$ of $e$ consists of extremal weight vectors and hence the grading
of the $n$-dimensional space   $\Lg_e$ by $\frac{1}{2} \alpha(h)$ coincides
with the values of $\mu$.

Since in this case $G_e$ is connected, we have $\Lg^{G_e} = \Lg^{\Lg_e}$ and the former has dimension $n$
since the moment map $T^*G/B \to \nilcone$ is birational, hence so does the latter. 
But then $\Lg_e = \Lg^{\Lg_e}$ for dimension reasons and since $e \in \Lg_e$ 
Finally, the discussion in \S\ref{exc_types} and the normality of $\nilcone$ 
give that the generalized exponents of $\Lg = V_\theta$ are given by the values of $\mu$.

The same proof applies to $V_\phi$ with $\Phi$ replaced by 
$\Phi_s$, the short roots of $\Lg_0$, since these are the nonzero  
weights of $V_\phi$, they correspond to one-dimensional weight spaces,
and the argument above and from \S \ref{gen_exps1}
shows that the kernel of $\mbox{ad}(e)$ on $V_\phi$ coincides with $V_\phi^{G_e} =
V_\phi^{\Lg_e}$ since all spaces have dimension $r = \dim V_\phi^T$.
We first learned of this result from \cite{Ion},
where it is proved in a different way.  Ion also credits Stembridge and Bazlov.

\begin{proposition}[Theorem 4.5 in \cite{Ion}, \cite{Vis}] \label{gen_exps_short}
Let $\Phi^+_s$ denote the short positive roots of $\Lg_0$.
Then the dual partition of the 
partition of $|\Phi^+_s|$ given by $$\# \{\alpha \in \Phi^+_s \ | \ \Ht(\alpha) = j  \}$$ for $j = 1, 2, \dots, \Ht(\phi)$
is equal to the generalized exponents 
$m^{\phi}_1 \leq m^{\phi}_1 \leq \dots \leq m^{\phi}_r$
of $V_{\phi}$.
\end{proposition}

\begin{example}
In type $F_4$, there are $2$ short roots of each height $1,2,3,$ and $4$, and $1$ short root of each height $5,6,7$, and $8$.  Therefore, the generalized exponents for $\phi$ are $4$ and $8$.  The usual exponents are $1,5,7,11$ using the heights for all the positive roots.  Combining with Proposition \ref{prop: simply-laced to non}, we know how the involution of $E_6$ that fixes $F_4$ acts on a set of $\epsilon$-stable fundamental invariants for $E_6$. 
\end{example}

%
%

\subsection{Type $A_{n}$} \label{gl}
For a generic matrix $X$ in $\mathfrak{gl}_{n+1}$, the entries of $X^{i}$ afford a representation isomorphic to $V_{\theta} \oplus \CC$ for $\Lg \simeq \mathfrak{sl}_{n+1}$.  These entries are not all zero on the nilcone of $\Lg$ when $1 \leq i \leq n$ since the regular element has Jordan form with one block of size $n+1$.  Consequently, they give the $n$ independent copies of $V_\theta$ in $\funnil$.  And since the derivatives of the $\tr(X^{i+1})$
and the entries of $X^{i}$ span the same space in $S \Lg^*$, the fundamental invariants can be taken to be $\tr(X^2), \dots \tr(X^{n+1})$ 
since
a set of invariant functions are fundamental invariants 
if and only if the vector of their derivatives are linearly independent when restricted to the $\nilcone$.
 Now, the matrix equation $X^j = X^{j-i}X^i $ for $j \geq i$ shows that the ideal in $S\Lg^*$ generated by the entries of 
$X^i$ contains the entries of $X^j$ for $j \geq i$.  Moreover, 
an ideal in $S \Lg^*$ generated by various copies of $V_{\theta}$, non-vanshing on $\nilcone$, and a set of fundamental invariants
is minimally generated by the entries of $X^i$ for some $i$ and $\tr(X^2), \dots, \tr(X^{i-1})$. 

\subsection{Type $C_{n}$}   \label{C}  
Let $X$ be a generic matrix of $\mathfrak{sp}_{2n} \subset \mathfrak{gl}_{2n}$.  Since eigenvalues of matrices in $\mathfrak{sp}_{2n}$ come in pairs of opposite sign, $\tr(X^i)$ vanishes on $\mathfrak{sp}_{2n}$ when $i$ is odd;  hence,
Proposition \ref{prop: simply-laced to non} and \S \ref{gl} imply that $\tr(X^2), \tr(X^4), \dots, \tr(X^{2n})$ are
a set of fundamental invariants.

\begin{proposition}  \label{C_restrict}
We have
\begin{enumerate}
\item  The entries of $X^{2i-1}$ restrict to $\mathbb C[\mathcal N_0]$ to give a nonzero copy of 
$V_\theta$ for $i =1, 2, \dots, n$.
\item  The entries of $X^{2i}$ restrict to $\mathbb C[\mathcal N_0]$ to give a nonzero copy of $V_{\phi}$
for $i =1, \dots, n\!-\!1$.
\item Any ideal in $S \Lg_0^*$ generated by various copies of $V_\phi$ or $V_\theta$ (non-vanishing on $\nilcone_0$) and a
set of fundamental invariants is minimally generated by a basis of the entries of $X^{j}$ for some $j<2n$ and 
$\tr(X^2), \tr(X^4), \dots, \tr(X^{2i})$ with $2i < j$.
\end{enumerate}
\end{proposition}
The last part follows from the matrix equation in $\S \ref{gl}$ together with the key fact that the entries of $X^i$, restricted to $\nilcone_0$, span a single irreducible representation.

\subsection{Types $B_n$ and $D_n$ }

An analogous story to that in \S \ref{non-simply} works for 
$\Lg \simeq \mathfrak{so}_{N} \subset \mathfrak{sl}_N$
and we get a version of Proposition \ref{C_restrict}.   We omit the proof since it is straightforward (and probably already known).
Here, $\mathfrak{sl}_N$ decomposes as $V_\theta \oplus V_{2\varpi_1}$ as $\mathfrak{so}_{N}$-module.
As is well-known, the Pfaffian $\mbox{Pf(X)}$ is a generator for $\mathfrak{so}_{N}$ when $N$ is even. 

\begin{proposition} \label{BD_restrict}
Let $X$ be a generic matrix of $\mathfrak{so}_{N} \subset \mathfrak{sl}_N$.   Then,
\begin{enumerate}
\item The functions $\tr(X^2), \tr(X^4), \dots, \tr(X^{2n-2})$, together with $\tr(X^{2n})$ (resp.\ $\mbox{Pf(X)}$), are complete set of fundamental invariants for $B_n$ (resp.\  $D_n$).
\item  The restriction of the entries of $X^{2i-1}$ to $\mathbb C[\mathcal N]$ gives nonzero copies of 
$V_\theta$ when $i =1, 2, \dots, n-1$ for $D_n$ and when $i =1, 2, \dots, n$ for $B_n$.
\item  The restriction of the entries of $X^{2i}$ to $\mathbb C[\mathcal N]$ gives all the nonzero copies of 
$V_{2\varpi_1}$ in $\funnil$
when $i =1, 2, \dots, n-1$ for $D_n$ and when $i =1, 2, \dots, n$ for $B_n$.
\item Any ideal in $S \Lg^*$ generated by the entries of $X^k$ for various $k$ 
is minimally generated by a basis of the entries of 
$X^{j}$ for some 
$j$.
\end{enumerate}
\end{proposition}

Finally, consider
$\Lg_0 \simeq \mathfrak{so}_{2n+1} \subset \Lg \simeq \mathfrak{so}_{2n+2}$ and let $X$ be a generic matrix of 
$\mathfrak{so}_{2n+2}$.  Since $\epsilon(\mbox{Pf(X)}) = -\mbox{Pf(X)}$ and $\epsilon(\tr(X^{2i})) = \tr(X^{2i})$, 
then by Proposition \ref{prop: simply-laced to non} the derivatives of $\mbox{Pf(X)}$ along $\Lg_0$ give the unique copy of $V_\phi$
in $\CC[\nilcone_0]$, in degree $n$.

\section{Explicit examples of invariants}  \label{invariants}

\begin{theorem}
Let $\phi:\Lg  \to {\mathfrak {gl}}( V_{\lambda})$ be a non-trivial highest weight representation of $\Lg$ of minimal dimension.   
Let $d$ be a degree for $\Lg$.  Let $X$ be the generic matrix in ${\mathfrak {gl}} ( V_{\lambda})$ with respect to some basis of 
$V_{\lambda}$.
Then the restriction of $\mbox{tr}(X^d)$ to $\Lg$ is a generator of $R$.  Moreover, if $\Lg$ is not of type $D_{2k}$, then this gives a complete set of fundamental invariants for $\Lg$.
\end{theorem}

\begin{proof}
We checked the result using Magma (see \cite{Ben_Thesis}) by restricting $\mbox{tr}(X^d)$ to the Kostant-Slodowy slice of $\Lg$ to the nilcone and using Kostant's slice result \cite{Kostant2}.  Namely, we observe that the restriction to the slice of $\mbox{tr}(X^d)$ contains a linear term.  The last statement follows since all the degrees are distinct in these cases.
\end{proof}

\begin{remark}
Normally to find a set of fundamental invariants we have to find a complete set, restrict to the Cartan subalgebra, and compute the Jacobian (see \cite{Lee}).  The method above has the advantage of being able to check one invariant at a time.  Thus, for example, we can check that $\mbox{tr}(X^8)$ and $\mbox{tr}(X^{12})$ are generators for $R$ in $E_8$, allowing perhaps a simpler starting point to the computer calculations in \cite{DPP}.   At the same time, this choice of invariants when restricted to the $\Lh$ seems more natural than that in \cite{Lee}.  For instance in $E_8$, we have that $\sum_{\alpha \in \Phi^+} \alpha^{d_i}$ for $1 \leq i \leq 8$ is a set of fundamental invariants for $W$.
\end{remark}

\begin{remark}
The first author \cite{Ben_Thesis} 
used this observation about invariants and Broer's description of the ideal defining the closure of the subregular nilpotent orbit $\0_{sr}$ to explicitly describe a generic singularity of $\overline \0_{sr}$.  This gives another way to obtain the result in \cite[\S 5.6]{FJLS}.
\end{remark}


\bibliographystyle{myalpha}
\bibliography{mini_paper}

\begin{thebibliography}{DCPP15}

\bibitem[Bro93]{Broer1}
B.~Broer.
\newblock Line bundles on the cotangent bundle of the flag variety.
\newblock {\em Invent. Math.}, 113(1):1--20, 1993.

\bibitem[Bro94]{Broer2}
B.~Broer.
\newblock Normality of some nilpotent varieties and cohomology of line bundles
  on the cotangent bundle of the flag variety.
\newblock In {\em Lie theory and geometry}, volume 123 of {\em Progr. Math.},
  pages 1--19. Birkh\"auser Boston, Boston, MA, 1994.

\bibitem[Bro99]{Broer_exponents}
A.~Broer.
\newblock Hyperplane arrangements, {S}pringer representations and exponents.
\newblock In {\em Advances in geometry}, volume 172 of {\em Progr. Math.},
  pages 83--93. Birkh\"auser Boston, Boston, MA, 1999.

\bibitem[DCPP15]{DPP}
C.~De~Concini, P.~Papi, and C.~Procesi.
\newblock The adjoint representation inside the exterior algebra of a simple
  {L}ie algebra.
\newblock {\em Adv. Math.}, 280:21--46, 2015.

\bibitem[Dem76]{Demazure}
M.~Demazure.
\newblock A very simple proof of {B}ott's theorem.
\newblock {\em Invent. Math.}, 33(3):271--272, 1976.

\bibitem[FJLS17]{FJLS}
B.~Fu, D.~Juteau, P.~Levy, and E.~Sommers.
\newblock Generic singularities of nilpotent orbit closures.
\newblock {\em Adv. Math.}, 305:1--77, 2017.

\bibitem[Ion04]{Ion}
B.~Ion.
\newblock The {C}herednik kernel and generalized exponents.
\newblock {\em Int. Math. Res. Not.}, (36):1869--1895, 2004.

\bibitem[Jan04]{jantzen}
J.~C. Jantzen.
\newblock Nilpotent orbits in representation theory.
\newblock In {\em Lie theory}, volume 228 of {\em Progr. Math.}, pages 1--211.
  Birkh\"auser Boston, Boston, MA, 2004.

\bibitem[Joh17]{Ben_Thesis}
B.~Johnson.
\newblock {\em Equations for nilpotent varieties and their intersections with
  Slodowy slices}.
\newblock 2017.
\newblock Thesis (Ph.D.)--University of Massachusetts Amherst.

\bibitem[Kac85]{kac}
V.~G. Kac.
\newblock {\em Infinite-dimensional {L}ie algebras}.
\newblock Cambridge University Press, Cambridge, second edition, 1985.

\bibitem[Kos59]{Kostant1}
B.~Kostant.
\newblock The principal three-dimensional subgroup and the {B}etti numbers of a
  complex simple {L}ie group.
\newblock {\em Amer. J. Math.}, 81:973--1032, 1959.

\bibitem[Kos63]{Kostant2}
B.~Kostant.
\newblock Lie group representations on polynomial rings.
\newblock {\em Amer. J. Math.}, 85:327--404, 1963.

\bibitem[Lee74]{Lee}
C.~Y. Lee.
\newblock Invariant polynomials of {W}eyl groups and applications to the
  centres of universal enveloping algebras.
\newblock {\em Canad. J. Math.}, 26:583--592, 1974.

\bibitem[OV90]{onish-vin:book}
A.~L. Onishchik and E.~B. Vinberg.
\newblock {\em Lie Groups and Algebraic Groups}.
\newblock Springer-Verlag, Berlin, 1990.

\bibitem[Ree98]{Reeder}
M.~Reeder.
\newblock Small modules, nilpotent orbits, and motives of reductive groups.
\newblock {\em Internat. Math. Res. Notices}, (20):1079--1101, 1998.

\bibitem[Ric87]{Richardson}
R.~W. Richardson.
\newblock Derivatives of invariant polynomials on a semisimple {L}ie algebra.
\newblock In {\em Miniconference on harmonic analysis and operator algebras
  ({C}anberra, 1987)}, volume~15 of {\em Proc. Centre Math. Anal. Austral. Nat.
  Univ.}, pages 228--241. Austral. Nat. Univ., Canberra, 1987.

\bibitem[Som03]{Sommers1}
E.~Sommers.
\newblock Normality of nilpotent varieties in {$E_6$}.
\newblock {\em J. Algebra}, 270(1):288--306, 2003.

\bibitem[Som05]{Sommers:very_even}
E.~Sommers.
\newblock Normality of very even nilpotent varieties in {$D_{2l}$}.
\newblock {\em Bull. London Math. Soc.}, 37(3):351--360, 2005.

\bibitem[Som09]{Sommers:cohomology}
E.~N. Sommers.
\newblock Cohomology of line bundles on the cotangent bundle of a
  {G}rassmannian.
\newblock {\em Proc. Amer. Math. Soc.}, 137(10):3291--3296, 2009.

\bibitem[Som16]{Sommers-Taiwan}
E.~Sommers.
\newblock Irreducible local systems on nilpotent orbits.
\newblock {\em Bulletin of the Institute of Mathematics Academia Sinica}, 2016.
\newblock https://arxiv.org/abs/1610.07645.

\bibitem[ST97]{Sommers-Trapa}
E.~Sommers and P.~Trapa.
\newblock The adjoint representation in rings of functions.
\newblock {\em Represent. Theory}, 1:182--189, 1997.

\bibitem[SYS80]{Saito}
K.~Saito, T.~Yano, and J.~Sekiguchi.
\newblock On a certain generator system of the ring of invariants of a finite
  reflection group.
\newblock {\em Comm. Algebra}, 8(4):373--408, 1980.

\bibitem[Vis06]{Vis}
S.~Viswanath.
\newblock A note on exponents vs root heights for complex simple {L}ie
  algebras.
\newblock {\em Electron. J. Combin.}, 13(1):Note 22, 5, 2006.

\bibitem[Wey89]{Weyman1}
J.~Weyman.
\newblock The equations of conjugacy classes of nilpotent matrices.
\newblock {\em Invent. Math.}, 98(2):229--245, 1989.

\bibitem[Wey02]{Weyman2}
J.~Weyman.
\newblock Two results on equations of nilpotent orbits.
\newblock {\em J. Algebraic Geom.}, 11(4):791--800, 2002.

\end{thebibliography}

\end{document}